\documentclass{amsart}

\usepackage[utf8]{inputenc} 

\usepackage{amsmath, amssymb, amsthm, amsfonts}
\usepackage[american]{babel}
\usepackage{hyperref}
\usepackage{color}
\usepackage{tikz-cd}
\usepackage{mathtools}
\usepackage{comment}
\usepackage{tikz}
\usepackage{rotating}

\usepackage[foot]{amsaddr}

\usetikzlibrary{arrows}

\newtheorem{theorem}{Theorem}[section]
\newtheorem{lemma}[theorem]{Lemma}
\newtheorem{definition}[theorem]{Definition}
\newtheorem{proposition}[theorem]{Proposition}
\newtheorem{corollary}[theorem]{Corollary}
\newtheorem{remark}[theorem]{Remark}

\numberwithin{equation}{section}

\newcommand{\N}{\mathbb{N}}
\newcommand{\Z}{\mathbb{Z}}

\newcommand{\R}{\mathbb{R}}
\newcommand{\C}{\mathbb{C}}
\newcommand{\Sc}{\mathcal{S}}
\newcommand{\F}{\mathcal{F}}
\newcommand{\T}{\mathbb{T}}

\newcommand{\J}{\mathfrak{J}}

\newcommand{\real}{\operatorname{Re}}

\DeclareMathOperator{\trace}{\rm tr\,} 

\newcommand{\supp}{\operatorname{supp}}

\title{Wellposedness of NLS in Modulation Spaces}
\author{Friedrich Klaus}
\address{Karlsruhe Institute of Technology, Englerstraße 2, 76131 Karlsruhe}
\email{friedrich.klaus@kit.edu}

\begin{document}
\maketitle

\begin{abstract}
	We prove new local and global well-posedness results for the cubic one-dimensional nonlinear Schrödinger equation in modulation spaces. Local results are obtained via multilinear interpolation. Global results are proven using conserved quantities based on the complete integrability of the equation, persistence of regularity, and by separating off the time evolution of finitely many Picard iterates.
\end{abstract}

\section{Introduction}

The last years have brought a great amount of interest towards the dynamics of the cubic nonlinear Schrödinger equation (NLS)
\begin{equation}\label{eq:NLS}
	\begin{cases}
		iu_t + u_{xx} = \pm 2|u|^2 u,\\
		u(0) = u_0,
	\end{cases}
\end{equation}
with initial data $u_0$ either decaying very slowly or not decaying at all. There are several ways to tackle this problem. In this paper we investigate the behavior of solutions to \eqref{eq:NLS} when the initial data is in a modulation space $M^s_{p,q}(\R)$ in one dimension.

Modulation spaces $M_{p,q}^s$ were introduced by Feichtinger \cite{feichtinger} and have by now been used in the study of various different PDE, see also \cite{ruzhansky,wangbook}. One of the reasons why they serve as an interesting space of initial data is because the decay of functions in modulation spaces $M^s_{p,q}$ is comparable to the one of functions in $L^p$. In particular, the spaces with $p = \infty$ include non-decaying initial data and provide them with an elegant function space framework. In contrast to $L^p$ and Besov spaces, the Schrödinger propagator is bounded on any modulation space $M^s_{p,q}$, and dispersive $L^\infty$ blow-up phenomena as constructed in \cite{dispersiveblowup} can be ruled out. A major open problem in this context is whether global in time existence in $M^s_{\infty,q}$ can be guaranteed for certain $s,q$. Just to name one of the many consequences an affirmative answer would have, this would solve the question whether a local solution to
\[
	u_0(x) = \cos(x) + \cos(\sqrt{2}x),
\]
can be continued globally. A unique local solution exists, e.g., by the work \cite{dss2} in a space of analytic functions, or by Picard iteration in the space $M_{\infty,1}$. While we are not able to give an answer to this question, we are able to prove global results with arbitrarily large $p < \infty$. Among other results (see Theorem \ref{thm:gwpplarge}) we will show: \emph{In the defocusing case, if $1 \leq p < \infty, 1 \leq q \leq \infty$ and if $s \geq 0$ is large enough, there is a unique global solution of \eqref{eq:NLS} in $M^s_{p,q}(\R)$}.

Local wellposedness results for nonlinear Schrödinger equations with initial data in modulation spaces have first been proven in \cite{wanghudzik2,wangzhaoguo,benyi,bhimani}. These results rely on boundedness of the Schrödinger propagator and an algebra property which holds either when $s \geq 0, q = 1$, or when $s > 1-1/q$. Later, the works \cite{guo,pattakos,kitgroup1} increased the range of admissible $p,q$ for $s=0$ using refined trilinear estimates for $p = 2, 2 \leq q < \infty$ and an infinite normal form reduction technique for $1 \leq q \leq 2, 2 \leq p \leq 10q'/(q'+6)$, respectively. Using complete integrability of the cubic one-dimensional NLS, Oh-Wang \cite{ohwang} showed the solutions of \cite{guo} to be global. Global solutions for initial data in $M_{p,p'}$ with $p$ sufficiently close to $2$ were constructed in \cite{kitgroup2}, though we note that these solutions were allowed take value in a different space $M_{\tilde p, \tilde q}$ for $t > 0$. Using decoupling techniques, Schippa \cite{schippa} recently proved $L^p$ smoothing estimates and extended the range of local wellposedness results for $p \in \{4,6\}$ and also, inspired by the work \cite{dodson}, gave global results for $q = 2, 2 \leq p < \infty, s > 3/2$. Finally we want to mention the preprint \cite{schippa2} in which Schippa very recently considered the energy-critical NLS with initial data in modulation spaces.

The goal of this work is twofold: On the one hand we want to give an overview of local well-posedness results and to unify the local results for $s = 0$. This is done by a Banach fixed point argument using multilinear interpolation on the estimates obtained in \cite{guo} and the trivial estimates for $q = 1$. From this we obtain local well-posedness in a range of $(p,q)$, comprising all of the aforementioned range for $s=0$ for which local well-posedness results were shown, except for the point $(p,q) = (4,2)$ from \cite{schippa}. The regularity $s = 0$ is sharp if we aim for analytic well-posedness by the considerations we provide in Section \ref{sec:illposedness}. 

On the other hand, we aim to extend the range of $(p,q)$ with global results, possibly assuming higher regularity of the initial data. To this end we first extend the almost conserved energies constructed in \cite{ohwang} to the range $p = 2, 1\leq q < 2$ and then use the principle of persistence of regularity to see that for a restricted range of $1 \leq p,q \leq 2$, the newly constructed local solutions are also global. Finally, we prove as in \cite{schippa,dodson} that in the defocusing case when we take $s \geq 1$, we obtain global solutions in $M^1_{p,1}$ for any $2 < p < \infty$. In fact, the same technique shows global well-posedness in $M^s_{p,q}$ for any $2 < p < \infty$ if $s > 2-1/q$ is large enough.

The paper is structured as follows: In Section \ref{sec:modulationspaces} we state basic facts on modulation spaces, in Section \ref{sec:qwp} we introduce the notion of quantitative well-posedness which gives the analytical framework to obtain our local well-posedness results in Section \ref{sec:lwp}. In Section \ref{sec:gwp} we prove the global results first for $p = 2$, then for $1 \leq p < 2$ and finally for $2 < p< \infty$. The well-posedness results are complemented by an illposedness result for $s < 0$ shown in Section \ref{sec:illposedness}.

\subsection*{Acknowledgements} The author wants to thank Peer Kunstmann and Robert
Schippa for helpful discussions. Funded by the Deutsche Forschungsgemeinschaft
(DFG, German Research Foundation) – Project-ID 258734477 – SFB 1173.

\section{Modulation Spaces}\label{sec:modulationspaces}

In this section we recall the definition of modulation spaces and state some results we need in later sections. Modulation spaces were introduced by Feichtinger \cite{feichtinger} in 1983 and have found growing interest in recent years. They can be introduced either via the short-time Fourier transform or equivalently via isometric decomposition on the Fourier side which also shows their close connection to Besov spaces. Modern introductions to modulation spaces are given in the books \cite{groechenig,wangbook}, and we also want to mention the PhD thesis \cite{chaichenetsthesis}. We refer to these for proofs of the following statements.

\begin{definition}
	The short-time Fourier transform of a function $f$ with respect to the window function $g \in \Sc(\R)$ is defined as
	\[
		V_g f(x,\xi) = \int_\R e^{-iy\xi}f(y)\bar g(y-x)\,dy.
	\]
	The modulation space norm of a function $f$ is defined as
	\[
		\|f\|^{\circ}_{M^s_{p,q}} = \Big(\int_{\R}\Big(\int_\R \big|V_g f(x,\xi)\big|^p\,dx\Big)^{\frac{q}{p}}\langle \xi\rangle^{sq}\,d\xi\Big)^{\frac{1}{q}}.
	\]
\end{definition}
With the usual modifications, this definition also includes $p,q = \infty$. We define the modulation space $M^{s}_{p,q}$ as those distributions in $\Sc'(\R)$ which have finite modulation space norm.

There is an equivalent norm on modulation spaces. Let $\rho \in \Sc(\R)$ be a smooth, symmetric bump function, that is $0 \leq \rho \leq 1$, $\rho(\xi) = 1$ if $|\xi|\leq 1/2$, $\rho(\xi) = 0$ if $|\xi| \geq 1$. Let
\[
	\rho_k(\xi) = \rho(\xi-k), \quad k \in \Z.
\]
Define $Q_0 = [-1/2,1/2)$ and $Q_k = k + Q_0$. Define
\[
	\sigma_k(\xi) = \rho_k(\xi)\big(\sum_{k\in\Z}\rho_k(\xi)\big)^{-1}, \quad k \in \Z.
\]
Then, for some $c > 0$, the functions $\sigma_k$ satisfy
\begin{equation}\label{eq:IDO}
	\begin{cases}
		|\sigma_k(\xi)| \geq c, \quad \forall \xi \in Q_k,\\
		\supp(\sigma_k) \subset \{|\xi-k| \leq 1\},\\
		\sum_{k\in \Z} \sigma_k(\xi) = 1, \quad \forall \xi \in \R,\\
		|D^{\alpha}\sigma_k(\xi)| \leq C_m, \quad \forall \xi \in \R, |\alpha| \leq m.
	\end{cases}
\end{equation}
\begin{definition}\label{def:IDO}
	Given a sequence of functions $\sigma_k$ satisfying \eqref{eq:IDO}, the sequence of operators
	\[
		\square_k = \F^{-1}\sigma_k \F, \quad k \in \Z,
	\]
	is called a family of isometric decomposition operators.
\end{definition}
\begin{definition}
	Given $p,q \in [1,\infty]$, $s \in \R$ and $(\square_k)_k$ a family of isometric decomposition operators. The modulation space norm with respect to $(\square_k)_k$ is defined as
	\[
		\|f\|_{M^{s}_{p,q}} = \big\|\langle k \rangle^s \|\square_k f\|_{L^p(\R)} \big\|_{\ell^q_k(\Z)}.
	\]
\end{definition}
It can be shown that for any family of isometric decomposition operators, $M^{s}_{p,q}$ can be equivalently characterized as those distributions in $\Sc'(\R)$ which have finite modulation space norm $\|\cdot\|_{M^{s}_{p,q}}$, and the norms $\|\cdot\|_{M^{s}_{p,q}}$ and $\|\cdot\|^{\circ}_{M^{s}_{p,q}}$ are equivalent. Moreover, the space of Schwartz functions $\Sc(\R)$ is dense in $M^{s}_{p,q}$ for any $p,q \in [1,\infty)$. If $p = \infty$, density fails. For instance we have continuous embeddings $C^2_b(\R) \subset M_{\infty,1} \subset C^0_b(\R)$.

As a consequence of Hölder's and Young's convolutional inequalities, we obtain bilinear bounds. These imply in particular that the spaces $M_{p,1}$ as well as $M_{p,q}\cap M_{\infty,1}$ are algebras under multiplication for all $p,q \in [1,\infty]$.
\begin{lemma}
	The following inequalities hold true: If $\frac{1}{p} = \sum_{i=1}^m \frac{1}{p_i}$ and $m-1+\frac{1}{q} = \sum_{i=1}^m \frac{1}{q_i}$ then
	\begin{equation}\label{eq:genhoelder}
		\Big\|\prod_{i=1}^m f_i\Big\|_{M_{p,q}} \lesssim \prod_{i=1}^m\|f_i\|_{M_{p_i,q_i}},
	\end{equation}
	and if $s \geq 0$, $\frac{1}{p} = \frac{1}{p_1} + \frac{1}{p_2}$, $1+\frac{1}{q} = \frac{1}{q_1}+\frac{1}{q_2} = \frac{1}{r_1}+\frac{1}{r_2}$ then
	\begin{equation}\label{eq:genhoelder2}
		\|fg\|_{M^{s}_{p,q}} \lesssim \|f\|_{M^{s}_{p_1,q_1}}\|g\|_{M_{p_2,q_2}} + \|f\|_{M_{p_1,r_1}}\|g\|_{M^s_{p_2,r_2}}.
	\end{equation}
\end{lemma}
\begin{proof}
	We give a short proof since \cite[Theorem 4.3]{chaichenetsthesis} only proves a similar statement. If we use the notation $l_1 + l_2 \approx k$ for $l_1 + l_2 = k + \{-1,0,1\}$, then
	\[
		\square_k (fg) = \square_k \big(\sum_{l_1}\square_{l_1}f\big)\big(\sum_{l_2}\square_{l_2}g\big) = \sum_{l_1 + l_2 \approx k} (\square_k\square_{l_1}f) (\square_k\square_{l_2}g).
	\]
	The operators $\square_{k}$ are bounded uniformly in $k$ on $L^{p_i}$. Hence
	\[
		\|\square_k (fg)\|_{L^p} \lesssim \sum_{l_1+l_2 \approx k} \|\square_{l_1}f\|_{L^{p_1}}\|\square_{l_2}f\|_{L^{p_2}}.
	\]
	Consequently, \eqref{eq:genhoelder} with $m = 2$ is obtained from Young's convolutional inequality. The case of general $m$ follows by induction. For \eqref{eq:genhoelder2} we use Peetre's inequality to see
	\[
		\langle k\rangle^s \|\square_k (fg)\|_{L^p} \lesssim  \sum_{l_1+l_2 \approx k} \langle l_1\rangle^s\|\square_{l_1}f\|_{L^{p_1}}\|\square_{l_2}f\|_{L^{p_2}} + \|\square_{l_1}f\|_{L^{p_1}}\langle l_2\rangle^s\|\square_{l_2}f\|_{L^{p_2}},
	\]
	and we conclude using Young's inequality.
\end{proof}

The bilinear bound allows to handle algebraic nonlinearities in nonlinear PDE. More complicated nonlinearities on the other hand can cause problems. In \cite{ruzhansky} Ruzhansky-Sugimoto-Wang raised the question whether an inequality of the form
\[
	\||f|^\alpha f\|_{M_{p,1}} \lesssim \|f\|_{M_{p,1}}^{\alpha+1}
\]
holds if $\alpha \in (0,\infty)\setminus 2\N$. This was answered negatively by Bhimani-Ratnakumar in \cite{bhimani}. In fact, they proved the stronger result that if a function $F: \R^2 \to \C$ operates in $M_{p,1}$ for some $1 \leq p \leq \infty$, then $F$ must be real analytic on $\R^2$. This also shows that in general, neither implication between $f \in M_{p,1}$ and $|f| \in M_{p,1}$ holds.

The following theorem shows how modulation spaces are nested. The first inclusion is a consequence of Bernstein's inequality and the embedding of $\ell^q$ spaces, whereas the second is a consequence of Hölder's inequality. 

\begin{theorem}[Embeddings]\label{thm:embedding}
    The following embeddings hold true:
    \begin{itemize}
        \item $M_{p_1,q_1}^{s_1} \subset M_{p_2,q_2}^{s_2}$ if $\quad p_1 \leq p_2, q_1 \leq q_2, s_1 \geq s_2$,
        \item $M_{p,q_1}^{s_1} \subset M_{p,q_2}^{s_2}$ if $\quad q_1 > q_2, s_1 > s_2, s_1 + \frac1{q_1}> s_2 + \frac1{q_2}$.
    \end{itemize}
\end{theorem}
The latter shows that we can trade regularity for $l^q$ summability. In one dimension, this gives for example $H^{1/2} \subset M^{2,1+}$ respectively $H^{1/2+} \subset M^{2,1}$. This is sharp since $M^{2,1} \subset L^\infty$ whereas $H^{1/2} \not\subset L^\infty$. On the other hand, $l^q$ summability does not gain regularity (see \cite{wanghudzik}):

\begin{lemma}
    We have that $M_{p, q} \not \subset B_{p, r}^{\varepsilon} \cup B_{\infty, \infty}^{\varepsilon}$ for any $0<\varepsilon \ll 1, 1\leq p, q, r \leq \infty$.
\end{lemma}
In particular an embedding of the form $M_{2,1} \subset H^{s}$ can never hold for positive regularity $s > 0$. The obstruction for this is $l^1 \not\subset l^2_s$ for $s > 0$. Indeed, one can just consider the sequence $a_n = 1/k^2$ if $n = 2^k$ and $a_n = 0$ else, i.e. spreading out mass in $l^2$ can be done without any problems - in contrast to $l^2_s$.

We recall some of the relations between modulation spaces, Besov spaces and $L^p$ spaces:

\begin{theorem}
	The following embeddings hold true:
	\begin{itemize}
		\item $M^{s}_{2,2} = H^s(\R)$, with equivalence of norms,
		\item $M_{p,1} \subset C_b^0(\R) \cap L^p(\R)$, if $\quad 1 < p \leq \infty$,
		\item $M_{p,p'} \subset L^p(\R)$, if $\quad 2 \leq p \leq \infty$,
		\item $M^{\sigma}_{p,q} \subset B_{p,q}$, if $\sigma = \max\Big(0,\frac{1}{\min(p,p')}-\frac{1}{q}\Big)$,
		\item $B_{p,q}^{\tau}\subset M_{p,q}$, if $\tau = \max\Big(0,\frac{1}{q} - \frac{1}{\max(p,p')}\Big)$.
	\end{itemize}
\end{theorem}
For example we see that $B_{2,1}^{\frac{1}{2}} \subset M_{2,1} \subset L^\infty \cap L^2$.

We will make use of the following result on complex interpolation of modulation spaces.

\begin{theorem}
    Let $p_0,p_1 \in [1,\infty]$ and $q_0,q_1 \in [1,\infty]$ such that $q_0 \neq \infty$ or $q_1 \neq \infty$. Let $s_0, s_1 \in \R$ and $\theta \in (0,1)$. Define
    \begin{align*}
        s &= (1-\theta)s_0 + \theta s_1,\\
        \frac{1}{p} &= \frac{1-\theta}{p_0} + \frac{\theta}{p_1}, \quad \frac{1}{q} = \frac{1-\theta}{q_0} + \frac{\theta}{q_1},
    \end{align*}
    with the usual convention in the extreme case $p_i,q_i = \infty$. Then,
    \begin{equation}
        \left[M_{p_{0}, q_{0}}^{s_{0}}(\mathbb{R}^{d}), M_{p_{1}, q_{1}}^{s_{1}}(\mathbb{R}^{d})\right]_{\theta}=M_{p, q}^{s}\left(\mathbb{R}^{d}\right),
    \end{equation}
    in the sense of equality of spaces and equivalence of norms.
\end{theorem}

Finally, since the decomposition on the Fourier side is uniform, there is no neat scaling relation for modulation spaces. Estimates still hold (see Theorem 3.2. in \cite{corderookoudjou}) and we shall use the ones for $p = 2$:

\begin{lemma}\label{lemma:scaling}
    We have the scaling inequalities
    \begin{align*}
        \| \psi(\lambda \cdot)) \|_{M_{2, q}} \lesssim\left\{\begin{array}{ll}
\lambda^{-1/2} \|\psi\|_{M_{2, q}}, \quad \text { if } \quad 1 \leq q \leq 2 \\
\lambda^{1/q-1} \|\psi\|_{M_{2, q}}, \quad \text { if } \quad 2 \leq q \leq \infty
\end{array}\right.
    \end{align*}
    and
    \begin{align*}
        \| \psi(\lambda \cdot)) \|_{M_{2, q}} \gtrsim\left\{\begin{array}{ll}
\lambda^{1/q-1} \|\psi\|_{M_{2, q}}, \quad \text { if } \quad 1 \leq q \leq 2 \\
\lambda^{-1/2} \|\psi\|_{M_{2, q}}, \quad \text { if } \quad 2 \leq q \leq \infty
\end{array}\right.
    \end{align*}
    for all $\lambda \leq 1$ and $\psi \in M_{2,q}$. Similarly,
    \begin{align*}
        \| \psi \|_{M_{2, q}} \lesssim\left\{\begin{array}{ll}
\lambda^{1/2} \|\psi(\lambda \cdot))\|_{M_{2, q}}, \quad \text { if } \quad 1 \leq q \leq 2 \\
\lambda^{1-1/q} \|\psi(\lambda \cdot))\|_{M_{2, q}}, \quad \text { if } \quad 2 \leq q \leq \infty
\end{array}\right.
    \end{align*}
    and
    \begin{align*}
        \| \psi \|_{M_{2, q}} \gtrsim\left\{\begin{array}{ll}
\lambda^{1-1/q} \|\psi(\lambda \cdot))\|_{M_{2, q}}, \quad \text { if } \quad 1 \leq q \leq 2 \\
\lambda^{1/2} \|\psi(\lambda \cdot))\|_{M_{2, q}}, \quad \text { if } \quad 2 \leq q \leq \infty
\end{array}\right.
    \end{align*}
    for all $\lambda \geq 1$ and $\psi \in M_{2,q}$.
\end{lemma}

If $u$ is a solution of cubic NLS, then so is $u_{\lambda}(x, t)=\lambda^{-1} u\left(\lambda^{-1} x, \lambda^{-2} t\right)$ for all $\lambda \in (0,\infty)$. Choosing $\lambda \geq 1$ we find that
	\begin{align*}
		\| u_\lambda(x,\lambda^2 t) \|_{M_{2, q}} \lesssim\begin{cases}
            \lambda^{-\frac{1}{2}} \|u(x,t)\|_{M_{2, q}}, \quad &\text { if } \quad 1 \leq q \leq 2, \\
            \lambda^{-\frac{1}{q}} \|u(x,t)\|_{M_{2, q}}, \quad &\text { if } \quad 2 \leq q \leq \infty,
            \end{cases}
	\end{align*}
	and
	\begin{align*}
		\| u_\lambda(x,\lambda^2t) \|_{M_{2, q}} \gtrsim\begin{cases}
            \lambda^{-\frac{1}{q}} \|u(x,t)\|_{M_{2, q}}, \quad &\text { if } \quad 1 \leq q \leq 2, \\
            \lambda^{-\frac{1}{2}} \|u(x,t)\|_{M_{2, q}}, \quad &\text { if } \quad 2 \leq q \leq \infty.
            \end{cases}
	\end{align*}
	
	In particular as long as $q < \infty$ we are in a subcritical range with respect to scaling.

\section{Quantitative Wellposedness}\label{sec:qwp}

Following \cite{bejenaru-tao} we quickly introduce the notion of quantitative wellposedness. While it is just a reformulation of the standard Picard iteration for homogeneous algebraic nonlinearities in a more quantitative fashion, it gives us the means to simply show linear and multilinear estimates and immediately obtain well-posedness. Our focus of application lies on the cubic NLS \eqref{eq:NLS} on the real line,
\begin{equation*}
	\begin{cases}
		iu_t + u_{xx} = \pm 2|u|^2 u,\\
		u(0) = f,
	\end{cases}
\end{equation*}
though the notion applies basically to any semilinear evolution equation with multilinear nonlinearity.

\begin{definition}\label{def:QWP}
    Let $L$ be a linear and $N_k$ be a $k$-multilinear operator. The equation
    \[ u = Lf + N_k(u,\dots,u)\]
    is called quantitatively wellposed in the spaces $D, X$ if the two estimates
    \begin{align}
        \|Lf\|_X &\leq C_1 \|f\|_{D},\label{eq:QWP1}\\
        \|N_k(u_1,\dots,u_k)\|_X &\leq C_2 \prod_{i=1}^k \|u_i\|_{X}\label{eq:QWP2}
    \end{align}
    hold for some constants $C_1, C_2 > 0$.
\end{definition}

As a consequence of polarization identities for real symmetric multilinear operators \cite{thomas}, in order to show an estimate of the form \eqref{eq:QWP2}, or more generally
\[
    \|N_k(u_1,\dots,u_k)\|_X \lesssim \prod_{i=1}^k \|u_i\|_{Y},
\]
it is enough to show the estimate
\[
    \|N_k(u,\dots,u)\|_X \lesssim \|u\|_{Y}^k.
\]
Indeed it is not hard to see via polarization that this implies
\[
    \|N_k(u_1,\dots,u_k)\|_X \lesssim \sum_{i=1}^k \|u_i\|_{Y}^k,
\]
and now putting $u_i = s_i \tilde u_i$ with $\prod s_i = 1$ and minimizing over $s_i$ proves the claim. This shows that for symmetric multilinear nonlinearities and norms that are invariant under complex conjugation, the contraction property of the corresponding operator in the Banach fixed point argument usually follows from being a self-mapping. In a similar manner one proves that the estimate
\[
    \|N_k(u,\dots,u)\|_X \lesssim \prod_{i=1}^k\|u\|_{Y_i}
\]
implies the estimate
\[
    \|N_k(u_1,\dots,u_k)\|_X \lesssim \sum_{\sigma \in S_k}\prod_{i=1}^k\|u_{\sigma(i)}\|_{Y_i},
\]
where $S_k$ denotes the permutation group of order $k$.

Denote by $B^X(R)$ the ball of radius $R$ in the space $X$. The reason for Definition \ref{def:QWP} is the following:

\begin{theorem}\label{thm:QWP}
    Let the equation
    \begin{equation}\label{eq:generalequation}
        u = Lf + N_k(u,\dots,u)
    \end{equation} 
    be quantitatively wellposed. Then there exist $\epsilon > 0$ and $C_0 > 0$ such that for all $f \in B^D(\epsilon)$ there is a unique solution $u[f] \in B^X(C_0 \epsilon)$ to \eqref{eq:generalequation}. In particular, $u$ can be written as an $X$-convergent power series for $f \in B^D(\epsilon)$,
    \begin{equation}
        u[f] = \sum_{n=1}^\infty A_n(f),
    \end{equation}
    where $A_n$ is defined recursively by
    \begin{equation*}
        A_1(f) = Lf, \qquad A_n(f) = \sum_{n_1 + \dots + n_k = n} N_k(A_{n_1}(f), \dots ,A_{n_k}(f)),
    \end{equation*}
    and satisfies for some $C_1,C_2 >0$,
    \begin{equation*}
    		\begin{cases}
    			A_n(\lambda f) = \lambda^n A_n(f)\\
    			\|A_n(f) - A_n(g)\|_{X} \leq C_1^n\|f-g\|_{D}(\|f\|_D + \|g\|_D)^{n-1},\\
    			\|A_n(f)\|_X \leq C_2^n \|f\|_D^n.
    		\end{cases}
    \end{equation*}
\end{theorem}


We will work in modulation spaces which do not admit homogeneous scaling, and are also above the scaling critical exponent for NLS. As a result, the bounds \eqref{eq:QWP1} and \eqref{eq:QWP2} will depend on the time variable $T$. This will show that a solution exists with guaranteed time of existence depending on $\|f\|_{D}$, and will result in a blow-up alternative later.

\begin{lemma}\label{lemma:blowupalternative}
    Let \eqref{eq:generalequation} be quantitatively wellposed in $D, X = X_T$, and assume that the constants in \eqref{eq:QWP1} respectively \eqref{eq:QWP2} are
    \[C_1 = c_1 \langle T\rangle^{\alpha_1}, \qquad C_2 = c_2 T^{\alpha_2}\langle T\rangle^{\alpha_3}.\]
    Then we may choose
    \[T \sim \min\big(\varepsilon^{-\beta_1},\varepsilon^{-\beta_2}\big), \qquad \beta_1 = \frac{k-1}{(k-1)\alpha_1 + \alpha_2 + \alpha_3},\qquad \beta_2 = \frac{k-1}{\alpha_2},\]
    as a guaranteed time of existence.
\end{lemma}
\begin{proof}
	If $\Phi(u) = Lf + N_k(u,\dots,u)$, then \eqref{eq:QWP1} and \eqref{eq:QWP2} give
	\[
		\|\Phi(u)\|_{X} \leq C_1\varepsilon + C_2 (C_0\varepsilon)^{k},
	\]
	which has to be smaller than $C_0\varepsilon$ for a contraction on $B^X(C_0\varepsilon)$. Taking $C_0 = 2C_1$ we need that
	\[
		2C_2 (2C_1\varepsilon)^{k-1} < 1,
	\]
	which amounts to
	\[
		T^{\alpha_2}\langle T\rangle^{\alpha_3 + \alpha_1(k-1)}\varepsilon^{k-1} \lesssim 1.
	\]
	When $\varepsilon$ is small, we can make $T$ large and $T \sim \langle T\rangle$ so that $\beta_1$ is the relevant exponent for $T$. When $\varepsilon$ is large, $\langle T \rangle \sim 1$ and we arrive at $\beta_2$. It is not hard to see that this also guarantees the Lipschitz bound to hold, and we obtain a contraction.
\end{proof}

We apply this general setting to cubic NLS and obtain:

\begin{definition}
    Let $D$ a Banach space of functions and let $S(t) = e^{it\partial_x^2}$. We call a function $u \in X_T \subset C^0([0,T],D)$ a (mild) solution of NLS if it solves the fixed point equation
    \begin{equation}\label{eq:fpnls}
        u = S(t)u_0 \mp 2i \int_0^t S(t-\tau)(|u|^2 u)(\tau) \, d\tau
    \end{equation}
    in $X_{T}$. The supremum of all such $T$ is called maximal time of existence and denoted by $T^*$.
\end{definition}

In the following we use the notation
\[
	N(u_1,u_2,u_3) = N_3(u_1,u_2,u_3) = 2i \int_0^t S(t-\tau)(u_1\bar u_2 u_3)(\tau) \, d\tau,
\]
and note that all local results we prove hold for both the focusing (minus sign in \eqref{eq:NLS}) and the defocusing (plus sign in \eqref{eq:NLS}) equation.

\begin{corollary}\label{cor:blowupalternative}
    Consider the Cauchy problem \eqref{eq:NLS} with initial data $f = u_0$ in a Banach space $D$. If the bounds
    \begin{align}
        \|S(t) u_0 \|_{X_T} &\lesssim \langle T\rangle^{\alpha_1}\|u_0\|_{D},\\
        \Big\|\int_0^t S(t-\tau)(u_1 \bar u_2 u_3)(\tau) \, d\tau\Big\|_{X_T} &\lesssim T^{\alpha_2}\langle T\rangle^{\alpha_3}\prod_{i=1}^3 \|u_i\|_{X_T},
    \end{align}
    hold for some $\alpha_1, \alpha_2,\alpha_3 > 0$, then for all $R > 0$ and $u_0 \in B^D(R)$ there exists $T > 0$ such that for all $T' < T$ there exists a unique solution $u \in X_{T'}$ to \eqref{eq:fpnls}. Moreover, the blowup-alternative
    \begin{equation}
        T^{*}<\infty \quad \Rightarrow \quad \limsup _{t \nearrow T^{*}}\|u(\cdot, t)\|_{D}=\infty
    \end{equation}
    holds.
\end{corollary}
\begin{proof}
	The existence and uniqueness follow from Theorem \ref{thm:QWP}. Assuming that $\|u(T)\|_{D} \leq C <  \infty$ with $T$ arbitrarily close to $T^*$, the assumptions from Lemma \ref{lemma:blowupalternative} are satisfied, hence there exists a small $\delta > 0$ such that \eqref{eq:NLS} can be solved on $[T^*,T^*+\delta)$, which contradicts the maximality.
\end{proof}

\section{Local Wellposedness via Multilinear Interpolation}\label{sec:lwp}
\subsection{The triangle $1/q \geq \max(1/p',1/p)$}

We recall the Strichartz estimates which lead to local wellposedness of \eqref{eq:NLS} in $L^2(\R)$ (see e.g. \cite{tao}).

\begin{lemma}[Strichartz estimates]
    Let $p \geq 2$. The following hold true:
    \begin{align}
        \|S(t)f\|_{L^p} &\lesssim |t|^{1/2-1/p}\|f\|_{L^{p'}},\\
        \|S(t)f\|_{L^2} &= \|f\|_{L^{2}}.
    \end{align}
    Moreover, call $(q,p)$ admissible if $2/q = 1/2-1/p$, $2 \leq p,q \leq \infty$ For all $(q,p)$ and $(\tilde q, \tilde p)$ admissible we have
    \begin{align}
        \|S(t)f\|_{L^q_t L^p_x} &\lesssim \|f\|_{L^2},\label{eq:strichartz}\\
        \|\int_0^t S(t-s)f(s)\|_{L^q_t L^p_x} &\lesssim \|f\|_{L^{\tilde q'}_t L^{\tilde p '}_x}.\label{eq:inhomstrichartz}
    \end{align}
\end{lemma}
Recall how this allows to prove local (and due to $L^2$ conservation also global) wellposedness of cubic NLS in $L^2(\R)$ by a fixed point argument: Let $X_T = L^\infty_t L^2_x([0,T]\times \R) \cap L^4_t L^\infty_x([0,T]\times \R)$. Then from Hölder's inequality,
\begin{align*}
    \|N(u_1,u_2,u_3)\|_{X_T} &= \|\int_0^t S(t-s) (u_1 \bar u_2 u_3)(s)ds\|_{X_T}\\
    &\lesssim \|u_1 \bar u_2 u_3 \|_{L^{8/7}_t L^{4/3}_x} \lesssim T^{1/2} \prod_{i=1}^3 \|u_i\|_{X_T}.
\end{align*}

Corollary \ref{cor:blowupalternative} together with $L^2$-conservation then gives global wellposedness in $L^2(\R)$. The space $L^\infty_t L^2_x([0,T]\times \R) \cap L^8_t L^4_x([0,T]\times \R)$ would have been enough for the iteration of the trilinear term, too.

The following estimates for the Schrödinger propagator in modulation spaces hold and are optimal with respect to the time dependence of the constant. A first version of them are proven in \cite{wangzhaoguo} in the case $p = 2$ which \cite{benyigroechenig} then extended for $p,q \in [1,\infty]$. Sharpness of the exponent for $p \in [1,2]$ was proven in \cite{corderonicola} and extended to $p \in [1,\infty]$ in \cite[Theorem 3.4]{chaichenetsthesis}.

\begin{lemma}
    Let $1 \leq p \leq \infty$ and $1 \leq q \leq \infty$. The following hold true:
    \begin{align}
        \|S(t)f\|_{M_{p,q}} &\lesssim (1+|t|)^{1/2} \|f\|_{M_{p,q}},\label{eq:modulationstrichartz1}\\
        \|S(t)f\|_{M_{p,q}} &\lesssim (1+|t|)^{-(1/2-1/p)} \|f\|_{M_{p',q}}, \quad \text{for} \quad p \geq 2,\label{eq:modulationstrichartz2}\\
        \|S(t)f\|_{M_{2,q}} &= \|f\|_{M_{2,q}},\label{eq:modulationstrichartz3}\\
        \|S(t)f\|_{M_{p,q}} &\lesssim (1+|t|)^{|1/2-1/p|} \|f\|_{M_{p,q}} \label{eq:modulationstrichartz4}.
    \end{align}
\end{lemma}

Note that \eqref{eq:modulationstrichartz4} is obtained by interpolating between \eqref{eq:modulationstrichartz1} with $p = 1, \infty$ and \eqref{eq:modulationstrichartz3}.

By Corollary \ref{cor:blowupalternative} we obtain local wellposedness in $M_{p,1}$ for all $1 \leq p \leq \infty$ with $X_T = C^0 M_{p,1}([0,T]\times \R)$ due to the trivial estimate
\begin{align*}
    \|N(u_1,u_2,u_3)\|_{X_T} &\lesssim \|\int_0^t S(t-s) (|u|^2 u)(s)ds\|_{X_T}\\
    &\lesssim T(1+T)^{|1/2-1/p|}\|u\|_{X_T}^3,
\end{align*}
which follows from the Banach algebra property of $M_{p,1}$.

Starting from the estimates for $M_{p,1}$ and $L^2 = M_{2,2}$ we use multilinear interpolation to obtain new local wellposedness results. The range of $p,q$ that can be reached as line segments between points $(1/p,1)$ and $(1/2,1/2)$ is exactly the triangle $1 \leq q \leq 2$, $1/q \geq \max(1/p',1/p)$, and this is where this simple multilinear interpolation works.


\begin{theorem}[\cite{berghloefstroem}, 4.4.1]\label{thm:multilinearinterpolation}
    Let $(A^{\nu}_0, A^{\nu}_1)_{(\nu = 1,\dots,n)}$ and $(B_0,B_1)$ be compatible Banach couples. Let $N: \sum_{1 \leq v \leq n}^{\oplus} A_0^\nu \cap A_1^\nu  \rightarrow B_0 \cap B_1$ be multilinear such that
    \begin{equation*}
        \begin{array}{l}
            \left\|N\left(a_{1}, \ldots, a_{n}\right)\right\|_{B_{0}} \leqslant M_{0} \prod_{\nu=1}^{n}\left\|a_{\nu}\right\|_{A_{0}^{\nu}}, \\
            \left\|N\left(a_{1}, \ldots, a_{n}\right)\right\|_{B_{1}} \leqslant M_{1} \prod_{\nu=1}^{n}\left\|a_{\nu}\right\|_{A_{1}^{\nu}}.
        \end{array}
    \end{equation*}
    Then $T$ can be uniquely extended to a multilinear mapping $\sum_{1 \leq v \leq n}^{\oplus} [A_0^\nu, A_1^\nu]_\theta  \rightarrow [B_0, B_1]_\theta$ with norm at most $M_0^{1-\theta} M_1^{\theta}$.
\end{theorem}

\begin{theorem}\label{thm:lwpuppertriangle}
    Let $1 \leq q \leq 2$, $1/q \geq \max(1/p',1/p)$. Then for any initial data $u_0 \in M^{p, q}$, there is a $T > 0$ and a unique solution $u$ to \eqref{eq:NLS} in
    \begin{equation}
        X^{p,q}_T = L^\infty_t M_{p, q}([0,T]\times\R) \cap L^{8/\theta}_t [M_{\tilde p,1},L^4]_\theta([0,T]\times \R).
    \end{equation}
    Here, the numbers $\theta \in [0,1]$ and $\tilde p \in [1,\infty]$ are determined by $1/ p = (1-\theta)/\tilde p + \theta/2$ and $1/ q = 1- \theta/2$. Moreover, either the solution $u$ exists globally in time, or there is $T^* < \infty$ such that
    \[
		\limsup_{t \to T^*} \|u(t)\|_{M_{p,q}} = \infty.
    \]
\end{theorem}
\begin{remark}
    Note that due to $M^{\tilde p,1} \subset L^\infty$ we have that $[M_{\tilde p,1},L^4]_\theta \subset L^{4/theta}$. This shows that the constructed solutions are also distributional.
\end{remark}
\begin{proof}
    Without loss of generality we assume $T \leq 1$. The assumptions on $\theta$ and $\tilde p$ imply that $M_{p,q} = [M_{\tilde p,1},L^2]_\theta$. We interpolate\footnote{Strictly speaking, instead of interpolating with the intersection we interpolate first on both spaces and then take the intersection. Interpolation of mixed-norm $L^p$ spaces was shown to work in \cite{benedekpanzone}. Since we can apply this to $\square_k f$ for each $k$ the same works if we consider mixed-norm combinations of $L^p$ and modulation spaces.} the linear estimates
    \begin{align*}
        \|S(t)u_0\|_{L^\infty_t L^2_x \cap L^8_t L^4_x} &\lesssim \|u_0\|_{L^2},\\
        \|S(t)u_0\|_{L^\infty_t M_{\tilde p,1}} &\lesssim \|u_0\|_{M_{\tilde p,1}},
    \end{align*}
    to obtain
    \begin{equation}\label{eq:Linterpolestimate}
        \|S(t)u_0\|_{X_T} \lesssim \|u_0\|_{M^{p,q}}.
    \end{equation}
    Moreover, the nonlinear estimates
    \begin{align*}
        \|N(u_1,u_2,u_3)\|_{L^\infty_t L^2_x \cap L^8_t L^4_x} &\lesssim T^{1/2}\prod_{i=1}^3 \|u_i\|_{L^8_t L^4_x},\\
        \|N(u_1,u_2,u_3)\|_{L^\infty_t M^{\tilde p,1}} &\lesssim T\prod_{i=1}^3 \|u_i\|_{L^\infty_t M_{\tilde p,1}},
    \end{align*}
    give, by Theorem \ref{thm:multilinearinterpolation},
    \begin{equation}\label{eq:NLinterpolestimate}
        \|N(u_1,u_2,u_3)\|_{X_T} \lesssim T^{1-\theta/2}\prod_{i=1}^3 \|u_i\|_{X_T}.
    \end{equation}
    The result now follows from Corollary \ref{cor:blowupalternative}.
\end{proof}

\subsection{The triangle $1/q > |1-2/p|$}

Using Bourgain space techniques, Guo showed local wellposedness of cubic NLS in $M_{2,q}, 2 \leq q <\infty$ \cite{guo}. Since his results were also derived from a trilinear estimate of the form \eqref{eq:QWP2}, we can use interpolation to get more wellposedness results in modulation spaces. The triangle $1/q > |1-2/p|$ is strictly larger than the triangle from the first section and can be obtained by means of interpolating between the three endpoints $M_{\infty,1}$, $M_{1,1}$ and $M_{2,\infty}$. Since the latter space contains the Dirac delta distribution and there is no local wellposedness theory for it, we have to exclude it and obtain wellposedness in a half-open triangle.

We introduce the $U^p$ and $V^p$ spaces in which wellposedness was achieved.

\begin{definition}
    A $U^p_t L^2_x((a,b)\times \R)$ atom is a function $A: (a,b) \to L^2$ of the form
    \[
    		A=\sum_{k=1}^{K} \chi_{\left[t_{k-1}, t_{k}\right)} \phi_{k},
	\]
	where $a = t_0 < \dots t_K = b$ and $(\phi_1, \dots, \phi_K) \in (L^2)^K$ which has unit norm in $l^p$, i.e. $\sum_i \|\phi_i\|_{L^2}^p = 1$. The space $U^p_t L^2_x$ is defined as the space of elements of the form $\sum_{j=1}^\infty \lambda_j A_j$, where $(\lambda_j) \in l^1$. It is equipped with the norm
    \begin{equation}
        \|u\|_{U^p} = \inf\{\|(\lambda_j)\|_{l^1}: u = \sum_{j=1}^\infty \lambda_j A_j \text{ for } A_j \text{ }U^p \text{ atoms}\}.
    \end{equation}
    The space $U^p_\Delta$ is defined as $S(\cdot)U^p_t L^2_x$ with norm 
    \begin{equation}
        \|u\|_{U^p_\Delta} = \|S(-t)u(t)\|_{U^p_t L^2_x}
    \end{equation}
\end{definition}
The spaces $U^2_t$ and its close cousin $V^2_t$ can be seen as refinements of Bourgain spaces in the case of $b = 1/2$, which satisfy $U^p_t \subset L^\infty_t$ for all $1 \leq p < \infty$. Indeed, the $X^{s,b}$ space would be defined by the norm $\|u\|_{X^{s,b}} = \|S(-t)u(t)\|_{H^b_t H^s_x}$. The usual Strichartz spaces are connected to the $U_\Delta^p$ spaces via
    \[
        \|v\|_{L_{t}^{p} L_{x}^{q}} \lesssim \|v\|_{U_{\Delta}^{p}}.
	\]
A proof of this can be found in \cite[Chapter 4]{kochtataruvisan} and we refer to this book as a reference for an introduction to these spaces.

\begin{theorem}[\cite{guo}]
    Let $2 < q < \infty$ and let $X^q_T$ denote the space of all tempered distributions $u$ such that the norm $\|u\|_{X^q} = \|\|\square_{n}u\|_{U_{\Delta}^{2}([0,T])}\|_{l^{q}}$ is finite. Then,
    \begin{align}
        \|N(u_1,u_2,u_3)\|_{X^{q}_T} \lesssim (T^{1/2} + T^{1/4} + T^{1/q^+})\prod_{i=1}^3\|u_i\|_{X^{q}_T}.
    \end{align}
\end{theorem}
This estimate gives local wellposedness in $X^q_T \subset L^\infty_t M^{2,q}([0,T]\times \R)$. Indeed, for the linear part the definition of $U_\Delta$ gives
\begin{equation}
    \begin{split}
    \|S(t)u_0\|_{X_q} &= \|\|\square_{n}S(t)u_0\|_{U_{\Delta}^{2}}\|_{l^{q}} = \|\|\square_{n}u_0\|_{U^{2}_t L^2_x}\|_{l^{q}} \\
    &\lesssim \|\|\square_{n}u_0\|_{L^2_x}\|_{l^{q}} = \|u_0\|_{M^{2,q}},
    \end{split}
\end{equation}

Since this result was only shown for $2 < q < \infty$ for the sake of simplicity let us define $X^{q}_T = X^{2,q}_T$ if $1 \leq q \leq 2$, where $X^{p,q}_T$ is as in Theorem \ref{thm:lwpuppertriangle}. Then we arrive at the following theorem which is proven analogous to Theorem \ref{thm:lwpuppertriangle}:

\begin{theorem}
    Let $1/q > |1-2/p|$. Then for any initial data $u_0 \in M_{p, q}$, there is a $T > 0$ and a unique solution $u$ to \eqref{eq:NLS} in
    \begin{equation}
        u \in Y_T^{p,q} = \begin{cases}[L^\infty_t M_{1,1},X^{\tilde q}]_\theta, \quad &\text{ if } 1 < p < 2,\\
        X^{q}_T, \quad &\text{ if } p = 2,\\
        [L^\infty_t M_{\infty,1},X^{\tilde q}]_\theta, \quad &\text{ if } 2 < p < \infty\end{cases}
    \end{equation}
    Here, $\tilde q$ is chosen such that
    \begin{equation}
        \frac{1}{q} = 1- \theta + \frac{\theta}{\tilde q}, \quad \frac{1}{p} = \begin{cases}1-\frac{\theta}{2},\quad   &\text{ if } p < 2,\\
        \frac{\theta}{2}, &\text{ if } p > 2.\end{cases}
    \end{equation}
    Moreover, either the solution $u$ exists globally in time, or there is $T^*< \infty$ such that
    \[
    		\limsup_{t \to T^*} \|u(t)\|_{M_{p,q}} = \infty.
    \]
\end{theorem}

\begin{remark}
	Taking into account the well-posedness in $M_{4,2}$ from \cite{schippa}, these results can be slightly strengthened to include the line $1/q = 1-2/p$, $4 \leq p \leq \infty$. Indeed, in \cite{schippa} the estimate
	\[
		\|S(t)f\|_{L^4([0,1]\times \R)} \lesssim \|f\|_{M_{4,2}}
	\]
	is shown to hold, which gives rise to an iteration in $L^\infty_t M_{4,2} \cap L^{\frac{24}{7}}_t L^4_x$. Interpolating the linear and the corresponding trilinear estimate with the estimates for $q = 1, p = \infty$ puts us into the setting of Corollary \ref{cor:blowupalternative}.
\end{remark}

\section{Global Wellposedness}\label{sec:gwp}

\subsection{Global Wellposedness if $p = 2$}

If $p = 2$ and $q > 2$, Oh-Wang \cite{ohwang} showed the existence of almost conserved quantities that are equivalent to the norms in the spaces $M_{p,q}$. To this end they used the complete integrability of cubic NLS via techniques from Killip-Visan-Zhang \cite{MR3820439} in combination with the Galilei transform. In this subsection, we extend these almost conserved quantities to the case $q \in (1,2)$ by using a weight with more decay, as it was done in \cite{MR3820439} for Besov spaces $B^s_{2,q}$.

First we state the necessary preliminaries from \cite{MR3820439}. Given an operator $A$ with continuous integral kernel $K(x,y)$, we define the trace
\begin{equation*}
	\trace (A) = \int_\R K(x,x)\,dx,
\end{equation*}
whenever it exists, and the Hilbert-Schmidt norm
\begin{equation*}
	\|A\|_{\J_2} = \int_{\R^2} |K(x,y)|^2 \,dxdy.
\end{equation*}
It can then be shown that for all $n \geq 2$,
\[
	|\trace(A_1 \dots A_n)| \leq \|A_1\|_{\J_2}\dots\|A_n\|_{\J_2}.
\]
We consider both focusing and defocusing cubic NLS in the form
\begin{equation}\label{eq:nlskvz}
	-iu_t = -u_{xx} \pm 2 |u|^2 u.
\end{equation}
Depending on the sign we have the following definition.

\begin{definition}
	The perturbation determinant $\alpha(\kappa,u)$ and its coefficients $\alpha_n(\kappa,u)$ are
	\begin{align*}
		\alpha(\kappa,u) &= \real \sum_{n=1}^\infty \frac{(\mp 1)^{n-1}}{n} \trace\big([(\kappa-\partial)^{-1/2}u(\kappa + \partial)^{-1}\bar u(\kappa-\partial)^{-1/2}]^n\big)\\
		&= \sum_{n=1}^\infty \alpha_{2n}(\kappa,u),
	\end{align*}
	where $\alpha_{2n}(\kappa,\lambda u) = \lambda^{2n}\alpha_{2n}(\kappa,u)$ for all $\lambda \in \R$.
\end{definition}

Absolute convergence of this series holds provided we can control norms sligthly stronger than $H^{-1/2}(\R)$. Define $a \sim b$ as $a\lesssim b$ and $a\gtrsim b$. Then:

\begin{lemma}[Lemma 4.1 in \cite{MR3820439}]\label{lemma:alphaestimate}
	Given $u \in \Sc(\R)$ and $\kappa > 0$, we have
	\begin{equation}
		\|(\kappa-\partial)^{-1/2}u(\kappa + \partial)^{-1/2}\|^2_{\J_2} \sim \int_\R \log\Big(4 + \frac{\xi^2}{\kappa^2}\Big) \frac{|\hat u(\xi)|^2}{(\xi^2 + 4\kappa^2)^{1/2}}d\xi.
	\end{equation}
	In particular for all $\delta > 0$,
	\begin{equation}
		|\alpha_{2n}(\kappa,u)| \lesssim \Big(\int_\R \frac{|\hat u(\xi)|^2}{(\xi^2 + 4\kappa^2)^{1/2-\delta}}\Big)^{n}.
	\end{equation}
\end{lemma} 

Even though we need $H^{-1/2+}(\R)$ regularity to control the series, the first coefficient in the expansion behaves similar to an $H^{-1}(\R)$ norm:

\begin{lemma}[Lemma 4.2 in \cite{MR3820439}]\label{lemma:secondterm}
	Given $u \in \Sc(\R)$ and $\kappa > 0$, we have
	\begin{equation}
		\alpha_2(\kappa,u) = \real \trace\big((\kappa - \partial)^{-1}u(\kappa + \partial)^{-1}\bar u\big) = \int_\R \frac{2\kappa |\hat u(\xi)|^2}{\xi^2 + 4\kappa^2}\,d\xi.
	\end{equation}
\end{lemma}

Most importantly, $\alpha(\kappa,u)$ is a conserved quantity for all $\kappa > 0$ whenever it is defined.

\begin{proposition}[Proposition 4.3 in \cite{MR3820439}]\label{prop:conservation}
	Given $u(t,x)$ a Schwartz-space solution of \eqref{eq:nlskvz} and $\kappa > 0$ large enough, we have
	\[
		\frac{d}{dt} \alpha(\kappa,u(t)) = 0.
	\]
\end{proposition}
In \cite{ohwang} the construction of the almost conserved quantity on the level of $M_{2,q}$ for $2 \leq q < \infty$ worked as follows: Combining Lemma \ref{lemma:secondterm}, Proposition \ref{prop:conservation} and invariance of \eqref{eq:nlskvz} under Galilean transformations, we obtain almost conservation of
\[
	\int_\R \frac{|\hat u(\xi)|^2}{(\xi-n)^2 + 1}\,d\xi,
\]
uniformly in $n$. 

Moreover, considering $\langle \xi\rangle^{-1/2-}$ instead of a compactly supported bump function for the uniform decomposition on the Fourier side in the definition of the modulation space norm gives an equivalent norm for $2 \leq q \leq \infty$. More precisely, if one defines
\[
	\|f\|_{MH^{\theta,q}} = \big(\sum_{n \in \Z} \|\langle \xi-n \rangle^{\theta}\hat f(\xi)\|_{L^2_\xi}^q\big)^{\frac{1}{q}},
\]
then for $\theta < -1/2$ and $2 \leq q \leq \infty$ one has
\[
	\|f\|_{MH^{\theta,q}} \sim \|f\|_{M_{2,q}}.
\]
We follow quickly the proof (see Lemma 1.2 in \cite{ohwang}) to motivate our next definition. The estimate ``$\gtrsim$'' is trivial since for $\sigma$ as in Definition \ref{def:IDO} we have $\sigma(\xi) \lesssim \langle \xi \rangle^{\theta}$. For the converse estimate, write $I_k = [k-1/2,k+1/2)$. Then,
\begin{align*}
	\|f\|_{MH^{\theta,q}} &= \Big(\sum_{n \in \Z} \Big(\int_\R\langle \xi-n \rangle^{2\theta}|\hat f(\xi)|^2\,d\xi\Big)^{\frac q2}\Big)^{\frac1q}\\
	&\sim \Big\| \sum_{k \in \Z} \langle k-n\rangle^{2\theta} \int_{I_k} |\hat f(\xi)|^2 \,d\xi\Big\|_{\ell^{q/2}_n}^{1/2}\\
	&\lesssim \big\|\langle n\rangle^{2\theta}\big\|_{\ell^1_n}^{1/2} \Big\| \int_{I_n} |\hat f(\xi)|^2 \,d\xi\Big\|_{\ell^{q/2}}^{1/2}\\
	&\lesssim \|f\|_{M_{2,q}}.
\end{align*}
We see that both the restriction $q \geq 2$ and $\theta < -1/2$ enter in the third line when Young's convolution inequality is used. If we have more decay available, i.e. if $\theta < -1$, we can also use the triangle inequality to get the full range of $q$.

\begin{lemma}
    If $\theta < -1$ and $1 \leq q \leq \infty$, we have
    \begin{equation*}
        \|f\|_{MH^{\theta,q}} \sim \|f\|_{M_{2,q}}.
    \end{equation*}
\end{lemma}
\begin{proof}
    Again ``$\gtrsim$'' follows immediately from $\sigma_n \lesssim \langle \cdot \rangle^{\theta}$. Now for the converse statement write
    \begin{align*}
    		\Big(\int_\R \langle \xi-n\rangle^{2\theta} |\hat f(\xi)|^2 \,d\xi\Big)^{\frac{1}{2}} &\sim \Big(\int_\R \sum_{l \in \Z} \sigma_l^2(\xi)\langle l-n\rangle^{2\theta} |\hat f(\xi)|^2 \,d\xi\Big)^{\frac{1}{2}}\\
    		&\leq \sum_{l \in \Z} \Big(\int_\R \sigma_l^2(\xi)\langle l-n\rangle^{2\theta} |\hat f(\xi)|^2 \,d\xi\Big)^{\frac{1}{2}}\\
    		&= \sum_{l \in \Z}\langle l-n\rangle^{\theta} \Big(\int_\R \sigma_l^2(\xi) |\hat f(\xi)|^2 \,d\xi\Big)^{\frac{1}{2}}.
    \end{align*}
    Thus,
    \begin{align*}
        \|u\|_{MH^{\theta,q}} &= \big(\sum_{n \in \Z} \|\langle \xi-n \rangle^{\theta}\hat f(\xi)\|_{L^2_\xi}^q\big)^{\frac{1}{q}}\\
        &\lesssim \Big\|\sum_{l \in \Z}\langle l-n\rangle^{\theta} \|\square_l f\|_{L^2}\|_{\ell^q}\\
        &\leq \|\langle n\rangle^{\theta}\|_{\ell^1_n} \|u\|_{M^{2,q}},
    \end{align*}
    by Young's inequality in the last step. Since $\theta < 1$, $\|\langle n\rangle^{\theta}\|_{\ell^1_n} <\infty$.
\end{proof}

From the form of $\alpha_2$ in Lemma \ref{lemma:secondterm} we see that we will get $\theta = -1$. By recombining $\alpha_2$ for different values of $\kappa$, we get more decay (see also Lemma 3.4 in \cite{MR3820439}).

\begin{definition}Define the weight function $w(\xi,\kappa)$ as
    \begin{equation}
        w(\xi,\kappa) = \frac{3\kappa^4}{(\xi^2 + \kappa^2)(\xi^2 + 4\kappa^2)}.
    \end{equation}
\end{definition}
A short calculation reveals that
\begin{equation*}
	w(\xi,\kappa) = 4 \frac{(\kappa/2)^2}{\xi^2 + \kappa^2} - \frac{\kappa^2}{\xi^2 + 4\kappa^2},
\end{equation*}
and hence
\begin{equation}
    4\kappa\alpha_2\Big(\frac\kappa2,u\Big) - \frac{\kappa}{2}\alpha_2(\kappa,u) = \int w(\xi,\kappa)|\hat u(\xi)|^2 \, d\xi.
\end{equation}
Correspondingly we define $\mathcal{F}(\tilde \square_n u)(\xi) = w(\xi-n,1)^{1/2}\hat u(\xi)$ and
\[
    		\|u\|_{\tilde M^{2,q}} = \|\|\tilde \square_n u\|_{L^2}\|_{l^q_n}.
\]

With these preparations we can prove:

\begin{theorem}\label{thm:gwpp=2}
	Let $q \in [1,\infty)$. There exists a constant $C = C(q)$ such that 
	\begin{equation}
		\|u(t)\|_{M_{2,q}} \leq
		\begin{cases}
		    C(1+\|u(0)\|_{M_{2,q}})^{\frac2q-1} \|u(0)\|_{M_{2,q}}, \quad &\text { if } \quad 1 \leq q \leq 2, \\
		    C(1+\|u(0)\|_{M_{2,q}})^{\frac{q}2-1} \|u(0)\|_{M_{2,q}}, \quad &\text { if } \quad 2 \leq q < \infty 
        \end{cases}
	\end{equation}
	for all $u \in \Sc(\R)$ solutions to the cubic NLS on $\R$.
\end{theorem}

\begin{proof}
	The case $2 \leq q < \infty$ was treated in \cite{ohwang}. In what follows we slightly modify its argument when $1 \leq q < 2$. Consider the case of small initial data in $M_{2,q}$ first and assume
	\begin{equation*}
		\|u(0)\|_{M_{2,q}} \leq \varepsilon \ll 1.
	\end{equation*}
	For $n \in \Z$, define $u_n(x,t) = e^{-inx + in^2 t} u(x-2nt,t)$ which satisfies $|\hat u_n(\xi,t)| = |\hat u(\xi+n,t)|$ and is a solution to cubic NLS as well.\\
	By Lemma \ref{lemma:alphaestimate} for any $\delta > 0$
	\begin{align*}
		\Big|\alpha\Big(u_n(t),\frac12\Big) - \alpha_2\Big(u_n(t),\frac12\Big)\Big| \lesssim \sum_{j=2}^\infty \left(\int_\R \frac{|\hat u(\xi,t)|^2}{(1+(\xi-n)^2)^{1/2-\delta}}\right)^j.
	\end{align*}
	Now for any $q \in (1,\infty)$ if $\delta$ is small enough,
	\begin{align*}
		\int_\R \frac{|\hat u(\xi,0)|^2}{(1+(\xi-n)^2)^{1/2-\delta}} &\sim \sum_k \frac{1}{(1+(k-n)^2)^{1/2-\delta}} \int_{I_k} |\hat u(\xi,0)|^2 d\xi\\
		&\lesssim \|u(0)\|_{M_{2,q}}^2,
	\end{align*}
	uniformly in $n \in \Z$. Indeed, if $2 < q < \infty$ we can employ Hölder's inequality provided with exponent $q/2$ if $\delta > 0$ is small enough. The case $1 \leq q \leq 2$ follows from $q = 2$ because of the embedding $M_{2,q} \subset L^2$. This shows that at time $t = 0$ the series for $\alpha$ is convergent. By continuity in time we can then choose a small time interval $0 \in I$ such that the series stays convergent, and
	\[
		\Big|\alpha\Big(u_n(t),\frac12\Big) - \alpha_2\Big(u_n(t),\frac12\Big)\Big| \lesssim \Big(\int_\R \frac{|\hat u(\xi,t)|^2}{(1+(\xi-n)^2)^{1/2-\delta}}\Big)^2,
	\]
	for all $t \in I$. The same argument works for $\kappa = 1$ instead of $\kappa = 1/2$.
	
	We calculate the difference of $\alpha$ and $\alpha_2$ by first making use of the above estimate, then localizing in Fourier space and then using Young's convolution inequality, with $I_k = [k,k+1)$,
	\begin{align*}
		\Big(\sum_{n \in \Z} \Big|\alpha\Big(u_n(t),\frac12\Big) - \alpha_2\Big(u_n(t),\frac12\Big)\Big|^{\frac{q}{2}}\Big)^{\frac1q} &\lesssim \Big(\sum_{n \in \Z} \Big(\int_\R \frac{|\hat u(\xi,t)|^2}{(1+(\xi-n)^2)^{1/2-\delta}}\Big)^q\Big)^{\frac1q}\\
		&\sim \Big\|\sum_{k \in \Z} \langle k - n \rangle^{-1+2\delta}\int_{I_k} |\hat u(\xi)|^2 \,d\xi\Big\|_{\ell^q}\\
		&\lesssim \|\langle k \rangle^{-1+2\delta}\|_{\ell^{1+}} \|u\|_{M_{2,2q-}}^2\\
		&\lesssim \|u\|_{M_{2,q}}^2,
	\end{align*}
provided $\delta > 0$ is small enough such that we can choose $q < 2q-$, and $q > 1$.

We use the definition of $\tilde M_{2,q}$, the subadditivity of the square root, Minkowski's inequality, Proposition \ref{prop:conservation} and the above estimate to find
	\begin{align*}
		\|u(t)\|_{\tilde M_{2,q}} &= \big\|\|\tilde \square_n u(t)\|_{L^2}\big\|_{\ell^q_n} = \Big\| \Big(4\alpha_2\Big(\frac12,u_n(t)\Big)-\frac12\alpha_2(1,u_n(t))\Big)^{\frac12}\Big\|_{\ell^q_n}\\
		&\leq \Big\| \Big|4(\alpha_2-\alpha)\Big(\frac12,u_n(t)\Big)-\frac12(\alpha_2-\alpha)(1,u_n(t))\Big|^{\frac12}\Big\|_{\ell^q_n} \\
		&\qquad + \Big\| \Big(4\alpha\Big(\frac12,u_n(t)\Big)-\frac12\alpha(1,u_n(t))\Big)^{\frac12}\Big\|_{\ell^q_n}\\
		&\leq 4\Big\|(\alpha_2-\alpha)\Big(\frac12,u_n(t)\Big)\Big\|_{\ell^{\frac q2}_n} + \frac12\Big\|(\alpha_2-\alpha)(1,u_n(t))\Big\|_{\ell^{\frac{q}{2}}_n} \\
		&\qquad + \Big\| \Big|4\alpha\Big(\frac12,u_n(0)\Big)-\frac12\alpha(1,u_n(0))\Big|^{\frac12}\Big\|_{\ell^q_n}\\
		&\leq \|u(0)\|_{\tilde M_{2,q}} + 4\sum_{s \in \{0,t\},\kappa \in \{1/2,1\}}\|(\alpha_2-\alpha)(\kappa,u_n(s))\|_{\ell^{\frac{q}{2}}_n}^{\frac12}\\
		&\leq \|u(0)\|_{\tilde M_{2,q}} + C(\|u(0)\|_{\tilde M_{2,q}}^2 + \|u(t)\|_{\tilde M_{2,q}}^2),
	\end{align*}
	for some constant $C > 0$. Using a continuity argument gives
	\begin{equation}
	    \|u(t)\|_{M_{2,q}} \lesssim \|u(0)\|_{M_{2,q}}
	\end{equation}
	if $\|u(0)\|_{M_{2,q}} \leq \varepsilon$ with $\varepsilon$ sufficiently small.\\
	For general initial data, we apply Lemma \ref{lemma:scaling} and the discussion thereafter. Consider $u_{\lambda}(x, t)=\lambda^{-1} u\left(\lambda^{-1} x, \lambda^{-2} t\right)$, which is a solution to NLS for all $\lambda \geq 1$. Then for $1 < q \leq 2$, we have
	\begin{equation*}
	    \left\|u_{\lambda}(0)\right\|_{M_{2, q}} \lesssim \lambda^{-\frac{1}{2}}\|u(0)\|_{M_{2, q}} \leq \varepsilon \ll 1
	\end{equation*}
	if $\lambda \sim (1+\|u(0)\|_{M_{2,q}})^2$. On the other hand,
	\begin{equation*}
	    \|u(t)\|_{M_{2, q}} \lesssim \lambda^{\frac{1}{q}}\left\|u_{\lambda}(\lambda^{2} t)\right\|_{M_{2, q}},
	\end{equation*}
	and so
	\begin{equation*}
	    \|u(t)\|_{M_{2, q}} \lesssim \lambda^{\frac{1}q - \frac12} \|u(0)\|_{M_{2,q}} \sim (1+\|u(0)\|_{M_{2,q}})^{\frac2q-1}\|u(0)\|_{M_{2,q}},
	\end{equation*}
	which finishes the proof if $1 < q < 2$.

This proof does not extend yet to $q = 1$ because the estimate of the tail does not have enough decay in $n$. The problem here is the coefficient $\alpha_4$ since for the tail of order homogeneity $6$ and more we can estimate with Young's inequality
\[
\begin{split}
	\sum_{n \in \Z} |\alpha(u_n)-\alpha_2(u_n)-\alpha_4(u_n)|^{1/2} &\lesssim \sum_{n \in \Z}\Big( \int_\R \frac{|\hat u(\xi,t)|^2}{(1+(\xi-n)^2)^{1/2-\delta}}\Big)^\frac{3}{2}\\
	&\sim \Big\|\sum_{k \in \Z} \langle k - n \rangle^{-1+2\delta}\int_{I_k} |\hat u(\xi)|^2 \,d\xi\Big\|_{\ell^{\frac{3}{2}}}^{\frac{3}{2}}\\
	&\lesssim \|\langle k \rangle^{-1+2\delta}\|_{\ell^{\frac{3}{2}}}^{\frac{3}{2}} \|u\|_{L^2}^3 \lesssim \|u\|_{L^2}^3,
\end{split}
\]
as long as $\delta$ stays small enough. To handle the sum
\[
	\sum_{n \in \Z} |\alpha_4(u_n)|^{1/2},
\]
we need to take a closer look at its structure. In \cite[Chapter 8.1]{MR3874652} Koch-Tataru prove a formula for $\tilde T_4$ which is related to $\alpha_4$ via $\alpha_4 = \real \tilde T_4(i\kappa)$ and reads
\[
	\tilde T_4(i\kappa) = \frac{i}{2\pi} \int_{\xi_1 + \xi_2 - \xi_3 - \xi_4 = 0} \frac{\real\big( \overline{\hat{u}(\xi_1)\hat u(\xi_2)}\hat u(\xi_3)\hat u(\xi_4)\big)}{(2i\kappa + \xi_1)(2i\kappa + \xi_3)(2i\kappa + \xi_4)}.
\]
This implies
\[
	\alpha_4 = \frac{1}{2\pi} \int_{\xi_1 + \xi_2 - \xi_3 - \xi_4 = 0} \frac{2\kappa(\xi_1\xi_3 + \xi_1\xi_4 + \xi_3\xi_4) - 8\kappa^3}{(4\kappa^2 + \xi_1^2)(4\kappa^2 + \xi_3^2)(4\kappa^2 + \xi_4^2)}\real\big( \overline{\hat{u}(\xi_1)\hat u(\xi_2)}\hat u(\xi_3)\hat u(\xi_4)\big).
\]
We concentrate on the part where there are frequencies in the numerator because the other part is more easily estimated. Now for example,
\begin{align*}
	&\int_{\xi_1 + \xi_2 - \xi_3 - \xi_4 = 0} \frac{|\xi_1\xi_3|}{(4\kappa^2 + \xi_1^2)(4\kappa^2 + \xi_3^2)(4\kappa^2 + \xi_4^2)}|\hat{u}(\xi_1)||\hat u(\xi_2)||\hat u(\xi_3)||\hat u(\xi_4)|\\
	&\quad \leq \Big\| \frac{|\xi_1|\hat u}{4\kappa^2+\xi_1^2}* \frac{|\xi_3|\hat u}{4\kappa^2+\xi_3^2} * \frac{\hat u}{4\kappa^2+\xi_4^2} * \hat u\Big\|_{L^\infty}\\
	&\quad \leq \Big\| \frac{|\xi|\hat u}{4\kappa^2+\xi^2}\Big\|_{L^2}^2\Big\| \frac{\hat u}{4\kappa^2+\xi^2}\Big\|_{L^1}\|\hat u\|_{L^1}\\
	&\quad \lesssim \Big\| \frac{\hat u}{\sqrt{4\kappa^2+\xi^2}}\Big\|_{L^2}^2\Big\| \frac{\hat u}{4\kappa^2+\xi^2}\Big\|_{L^1}\|u\|_{M_{2,1}}.
\end{align*}
Here we used Young's convolution inequality and the fact that
\[
	\int_\R |\hat u(\xi)|d\xi = \sum_{k \in \Z}\int_{I_k} |\hat u(\xi)|d\xi \leq \sum_{k \in \Z}\|\hat u\|_{L^2(I_k)} = \|u\|_{M_{2,1}}.
\]
Thus to bound $\sum_{n \in \Z} |\alpha_4(u_n)|^{1/2}$ we estimate
\begin{align*}
	 \sum_{n \in \Z} &\Big(\|u_n\|_{M_{2,1}}\int \frac{|\hat u_n(\xi)|^2}{4\kappa^2 + \xi^2}\,d\xi\int \frac{|\hat u_n(\xi)|}{4\kappa^2 + \xi^2}\,d\xi\Big)^{\frac12}\\
	&\sim \sum_{n \in \Z}\|u\|_{M_{2,1}}^{\frac12}\Big(\sum_k\frac{\int_{I_k}|\hat u|^2}{4\kappa^2+(k-n)^2}\,d\xi \sum_l \frac{\int_{I_k} |\hat u|}{4\kappa^2 + (l-n)^2}\, d\xi\Big)^{\frac12}\\
	&\leq \|u\|_{M_{2,1}}^{\frac12} \Big\|\Big(\sum_k\frac{\int_{I_k}|\hat u|^2}{4\kappa^2+(k-n)^2}\,d\xi\Big)^{\frac12}\Big\|_{\ell^2_n}\Big\|\Big(\sum_l \frac{\int_{I_k} |\hat u|}{4\kappa^2 + (l-n)^2}\, d\xi\Big)^{\frac12}\Big\|_{\ell^2_n}\\
	&\leq \|u\|_{M_{2,1}}^{\frac12}\big(\|\hat u\|_{L^2}^2\sum_k \frac{1}{4\kappa^2 + k^2}\big)^{\frac12}\big(\|\hat u \|_{L^1}\sum_l \frac{1}{4\kappa^2 + l^2}\big)^{\frac12}\\
	&\lesssim \kappa^{-1}\|u\|_{M_{2,1}} \|u\|_{L^2}.
\end{align*}
In the first line we estimated with the inequality from above, then we discretized in Fourier space, then we estimated via Hölder and Young's convolution inequality, and finally we used again that the $L^1$ norm of the Fourier transform is bounded by the $M_{2,1}$ norm and that the scaling behavior of the sums is $\kappa^{-1/2}$.

Arguing as before, we also obtain the case $q = 1$.
\end{proof}

\subsection{Global Wellposedness if $p < 2$}

If $p < 2$, the spaces $M_{p,q}$ are contained in $M_{2,q}$ and we expect an upgrade to a global result with the use of the principle of persistence of regularity (see e.g. \cite{tao}). We use the following version of Gronwall's inequality:

\begin{lemma}\label{lemma:gronwall}
	Let $u, \alpha,\beta: [a,b] \to \R$ be continuous with $\beta \geq 0$. Assume that for all $t \in  [a,b]$,
	\[
		u(t) \leq \alpha(t) + \int_a^t \beta(s)u(s)\,ds.
	\]
	Then also
	\[
		u(t) \leq \alpha(t) + \int_a^t \alpha(s)\beta(s)e^{\int_s^t \beta(s')\,ds'}\,ds.
	\]
\end{lemma}

The following blow-up alternative is easily obtained:

\begin{lemma}\label{lemma:gwprightup1}
     If for all $T>0$,
     \[
     	\sup_{t \in [0,T]} \|u(t)\|_{M_{\infty,1}} < \infty,
     \]
     and if cubic NLS is locally well-posed in $M_{p,q}^s(\R)$ for some $1 \leq p,q \leq \infty$, $s \geq 0$, then it is also globally wellposed in this space.
\end{lemma}
\begin{proof}
By Corollary \ref{cor:blowupalternative} we have to show that the $M^s_{p,q}(\R)$ norm cannot blow up. Now $u$ solves
\begin{equation}
    u(t) = S(t) u_0 + 2i\int_0^t S(t-s) |u|^2 u(s) ds,
\end{equation}
and hence if $0 \leq t \leq T$, estimating with \eqref{eq:genhoelder2},
\begin{equation}
    \|u(t)\|_{M^s_{p,q}} \lesssim_T \|u_0\|_{M^s_{p,q}} + \|u\|^2_{L^\infty([0,T],M_{\infty,1})}\int_0^t \|u(s)\|_{M^s_{p,q}} ds.
\end{equation}
Using the assumption $\|u\|^2_{L^{\infty}([0,T],M_{\infty,1})} \leq C$ we can use Gronwall's inequality and conclude.
\end{proof}

Lemma \ref{lemma:gwprightup1} tells us that the $M_{\infty,1}$ norm is a controlling norm in this setting. This shows that when $1\leq p \leq 2$, $1 \leq q \leq \infty$ and $s$ is high enough, not only the question of local but also of global well-posedness becomes trivial: From the embedding $H^{1/2+} \subset M_{2,1}$ and the construction of conserved quantities adapted to $H^s$ for any $s > -1/2$ \cite{MR3874652,MR3820439} we find global in time bounds in $M_{\infty,1}$ if we just embed into $H^{1/2+}$. In the spaces $M_{p,1}$ with $1 \leq p \leq 2$ we also find global well-posedness due to Theorem \ref{thm:gwpp=2}. The case $p > 2$ is more complicated and treated below.

For $s = 0$ and general $1 < q < \infty$, we obtained the local well-posedness via interpolation. In the upper triangle $1/q \geq \max(1/p',1/p)$ the Picard iteration space was
\[
	X_T^{p,q} = L^\infty_t M_{p,q}([0,T]\times\R) \cap L_t^{\frac{8}{\theta}}[M_{\tilde p,1},L^4]_\theta([0,T]\times \R).
\]
Note that we could equally well have iterated in
\[
	\tilde X_T^{p,q} = L^\infty_t M_{p,q}([0,T]\times\R) \cap L_t^{\frac{4}{\theta}}[M_{\tilde p,1},L^\infty]_\theta([0,T]\times \R),
\]
because the Strichartz estimates holds true up to $L^4_t L^\infty_x$ in one dimension. With this at hand, we can prove:

\begin{lemma}\label{lemma:gwprightup}
    Cubic NLS is globally wellposed in $M_{p,q}(\R)$, $1 \leq p < 2$, $1/q \geq 1/p$.
\end{lemma}
\begin{proof}
We interpolate the multilinear estimates
    \begin{align*}
        \|u_1 \bar u_2 u_3 \|_{M_{\tilde p,1}} &\lesssim \|u_1\|_{M_{\infty,1}} \|u_2\|_{M_{\infty,1}} \|u_3\|_{M_{\tilde p,1}},\\
        \|u_1 \bar u_2 u_3 \|_{L^2} &\leq \|u_1\|_{L^\infty}\|u_2\|_{L^\infty}\|u_3\|_{L^2}
    \end{align*}
to obtain
\begin{equation}
	\|u_1 \bar u_2 u_3 \|_{M_{p,q}} \lesssim \|u_1\|_{[M_{\infty,1},L^\infty]_\theta} \|u_2\|_{[M_{\infty,1},L^\infty]_\theta} \|u_3\|_{M_{p,q}},
\end{equation}
where $p,q,\theta$ are exactly as in Theorem \ref{thm:lwpuppertriangle}. This shows
\begin{align*}
	\Big\|\int_0^t S(t-s)|u|^2 u\,ds\Big\|_{M^{p,q}} \lesssim \int_0^t \|u\|^2_{[M_{\infty,1},L^\infty]_\theta} \|u\|_{M_{p,q}}\,ds,
\end{align*}
and we can conclude as in Lemma \ref{lemma:gwprightup1} if we know that $\|u\|_{L^2([0,T],[M_{\infty,1},L^\infty]_\theta)}$ remains finite. Now with continuous inclusion with $T$-dependent constants,
\[
\begin{split}
	[L^\infty([0,T],M_{2,1}),L^4([0,T],L^\infty)]_\theta &\subset [L^2([0,T],M_{2,1}),L^2([0,T],L^\infty)]_\theta \\
	&= L^2([0,T],[M_{2,1},L^\infty]_\theta) \\
	&\subset L^2([0,T],[M_{\infty,1},L^\infty]_\theta).
\end{split}
\]
Since we could have chosen the left-hand side as the iteration space in Theorem \ref{thm:lwpuppertriangle} we conclude that the solution has locally bounded norm in this space with estimate
\[
	\|u\|_{[L^\infty([0,1],M_{2,1}),L^4([0,1],L^\infty)]_\theta} \lesssim \|u_0\|_{M^{2,q}}.
\]
Note that $p < 2$, hence $M^{p,q} \subset M^{2,q}$. The $M^{2,q}$ norm does not blow up, hence the norm on the left-hand side does not blow up even if we replace $[0,1]$ by a time interval $[0,T]$ as we can just glue together solutions.
\end{proof}

\subsection{Global Wellposedness if $p > 2$}
In the case $u_0 \in M_{p,1}$ with $2 < p < \infty$, we want to use techniques inspired by \cite{dodson}. Similar results were obtained for $p = 4$ and $p = 6$ in \cite{schippa}. Note though that the spaces $M^s_{4,2}$ and $M^s_{6,2}$ with $s > 3/2$ embed into $M^1_{4,2}$ and $M^1_{6,2}$ in which we will prove global wellposedness. The goal is to make use of the fact that there is a number $N$ such that for $n \geq N$, the $n$th Picard iterates will be in an $L^2$ based space. Indeed, if we keep the notation from Theorem \ref{thm:QWP}, then by the multilinear estimate \eqref{eq:genhoelder},
\[
	\|A_3(u_0)\|_{L^\infty([0,1],M_{2,1})} \lesssim \||S(t)u_0|^2S(t)u_0\|_{L^\infty([0,1],M_{2,1})} \lesssim \|u_0\|^3_{M_{6,1}},
\]
and similarly for each natural number of the form $4n+2, n \in \N_0$, we have
\begin{equation}\label{eq:estimateAn}
	\|A_{2n+1}(u_0)\|_{L^\infty([0,1],M_{2,1})} \lesssim_n \|u_0\|^{2n+1}_{M_{4n+2,1}}.
\end{equation}
More generally, we find:
\begin{lemma}\label{lemma:picarditerates}
	Given odd natural numbers $k_1, k_2, k_3 \in \N$ and $2m + 1 = k_1 + k_2 + k_3$, and $n \in \N$ with $m \geq n$, the following estimates hold:
	\begin{align}
		\|N(A_{k_1},A_{k_2},A_{k_3})\|_{L^\infty([0,T],M_{p,1})} &\lesssim_m T^{m}\langle T\rangle^{m+1/2}\|u_0\|_{M_{p(2m+1),1}}^{2m+1}\label{eq:picard1}\\
		\|A_{2n+1}\|_{L^\infty([0,1],M_{2,1})} &\lesssim_n T^{n}\langle T\rangle^{n+1/2}\|u_0\|^{2n+1}_{M_{4n+2,1}}\label{eq:picard2}\\
		\|A_{2m+1}\|_{L^\infty([0,1],M_{2,1})} &\lesssim_m T^{m}\langle T\rangle^{m+1/2}\|u_0\|^{2n+1}_{M_{4n+2,1}}\|u_0\|^{2(m-n)}_{M_{\infty,1}}\label{eq:picard3}.
	\end{align}
\end{lemma}
\begin{proof}
	We use the estimate for $0 \leq t \leq T$
	\[
	\begin{split}
		\|N(A_{k_1},A_{k_2},A_{k_3})\|_{M_{p,1}} &= \Big\|\int_0^t S(t-s)A_{k_1}\bar A_{k_2}A_{k_3}\,ds\Big\|_{M_{p,1}} \\
		&\lesssim T\langle T\rangle^{1/2}\|A_{k_1}\|_{M_{p_1,1}}\|A_{k_2}\|_{M_{p_2,1}}\|A_{k_3}\|_{M_{p_3,1}},
	\end{split}
	\]
	provided $\sum_i 1/p_i = 1/p$. Plugging in the definition of $A_{k_i}$ from Theorem \ref{thm:QWP} iteratively shows that after $m$ iterations we arrive at
	\[
		\|A_{2m+1}(u_0)\|_{M_{p,1}}+\|N(A_{k_1},A_{k_2},A_{k_3})\|_{M_{p,1}} \lesssim_n T^m \langle T\rangle^{\frac m2}\|Lu_0\|_{M_{(2m+1)p,1}}^{2m+1},
	\]
	if $k_1 + k_2 + k_3 = 2m+1$. Together with 
	\[
		\|Lu_0\|_{M_{(2m+1)p,1}}^{2m+1} \lesssim \langle T\rangle^{\frac{2m+1}2} \|u_0\|_{M_{(2m+1)p,1}}^{2m+1},
	\]
	\eqref{eq:picard1} and \eqref{eq:picard2} follow. To prove \eqref{eq:picard3} we additionally use
	\[
		\|uvw\|_{M_{4n+2,1}} \lesssim \|u\|_{M_{\infty,1}}\|v\|_{M_{\infty,1}}\|w\|_{M_{4n+2,1}},
	\]
	once we reached $p = 4n+2$ in the iteration.
\end{proof}

As is shown for the usual Picard iteration (see for example Theorem 3 in \cite{bejenaru-tao}), and because there is no loss in the constant from Hölder's inequality \eqref{eq:genhoelder}, the constant in \eqref{eq:estimateAn} grows at most exponentially in $n$ meaning that we are able to sum the remainder term. This motivates that we will be able to construct a solution of NLS of the form
\begin{equation}\label{eq:formofu}
	u(t) = \sum_{k=1}^{2n-1} A_k(u_0) + v = \tilde u + v,
\end{equation}
where
\[
\begin{split}
	\tilde u \in C^0([0,T],M_{4n+2,1})\quad \text{and} \quad v \in C^0([0,T],M_{2,1}).
\end{split}
\]
If $u$ has the form \eqref{eq:formofu} and solves NLS then $v$ will solve the difference NLS
\begin{equation}\label{eq:nlsforv2}
	\begin{cases}iv_t + v_{xx} &= |u|^2u - G(t),\\
	v(0) &= 0
	\end{cases}
\end{equation}
where $G(t)$ is given by
\[
	G(t) = i\tilde u_t + \tilde u_{xx} = \sum_{k=3}^{2n-1} \sum_{k_1 + k_2 + k_3 = k} A_{k_1}(u_0)\bar A_{k_2}(u_0)A_{k_3}(u_0).
\]
As a fixed point equation this equation reads
\begin{equation}\label{eq:nlsforv}
	v(t) = N\Big(v+\sum_{k=1}^{2n-1} A_k(u_0), v+\sum_{k=1}^{2n-1} A_k(u_0),v+\sum_{k=1}^{2n-1} A_k(u_0)\Big) - \sum_{k=3}^{2n-1} A_k(u_0).
\end{equation}
The existence and uniqueness issue for $v$ is covered in the following lemma.
\begin{lemma}\label{lemma:localmodulation}
	Let $u_0 \in M_{4n+2,1}$. There exists $T>0$ and a solution $v \in C^0([0,T],M_{2,1})$ of \eqref{eq:nlsforv}. The solution is unique in $L^\infty([0,T],M_{4n+2})$. If $T^*$ denotes its maximal time of existence, then either $T^* = \infty$ or
	\[
		\limsup_{t \to T^*} \|v(t)\|_{M_{4n+2,1}} = \infty.
	\]
\end{lemma}
\begin{proof}
	We ignore permutations of the arguments of $N$ and rewrite \eqref{eq:nlsforv} as
	\[
	\begin{split}
		v(t) &= N\Big(v+\sum_{k=1}^{2n-1} A_k(u_0), v+\sum_{k=1}^{2n-1} A_k(u_0),v+\sum_{k=1}^{2n-1} A_k(u_0)\Big) - \sum_{k=3}^{2n-1} A_k(u_0)\\
		&= N(v, v,v) + N\Big(v, v,\sum_{k=1}^{2n-1} A_k(u_0)\Big)+ N\Big(v,\sum_{k=1}^{2n-1} A_k(u_0),\sum_{k=1}^{2n-1} A_k(u_0)\Big)\\
		&\quad + N\Big(\sum_{k=1}^{2n-1} A_k(u_0),\sum_{k=1}^{2n-1} A_k(u_0),\sum_{k=1}^{2n-1} A_k(u_0)\Big) - \sum_{k=3}^{2n-1} A_k(u_0).
	\end{split}
	\]
	If we define the function in the last line to be $F(t,x)$, then we can show
	\begin{equation}\label{eq:estimateF}
		\|F\|_{L^\infty([0,T],M_{2,1})} \lesssim T^n\langle T\rangle^{n+1/2}\|u_0\|_{M_{4n+2,1}}^{2n+1} + T^{3n-2}\langle T\rangle^{3n-3/2}\|u_0\|_{M_{4n+2,1}}^{6n-3}.
	\end{equation}
	Indeed, we rewrite
	\[
	\begin{split}
		N\Big(\sum_{k=1}^{2n-1} A_k(u_0)&,\sum_{k=1}^{2n-1} A_k(u_0),\sum_{k=1}^{2n-1} A_k(u_0)\Big) \\
		&\qquad = \sum_{m=1}^{2n-1}\sum_{k_1+k_2+k_3 = m} N(A_{k_1}(u_0),A_{k_2}(u_0),A_{k_3}(u_0)) + F(t,x)\\
		&\qquad =\sum_{k=3}^{2n-1} A_k(u_0) + F(t,x),
	\end{split}
	\]
	and use Lemma \ref{lemma:picarditerates} to estimate. In the same fashion, we find
	\[
		\Big\|\sum_{k=1}^{2n-1} A_k(u_0)\Big\|_{L^\infty([0,T],M_{\infty,1})} \lesssim \langle T\rangle^{1/2}\|u_0\|_{M_{4n+2,1}} + T^{n-1}\langle T\rangle^{n-1/2}\|u_0\|_{M_{4n+2,1}}^{2n-1}.
	\]
	This shows that if $\Phi(v)$ is the right-hand side in \eqref{eq:nlsforv}, and if $\|v\|_{L^\infty([0,T],M_{2,1})} \leq R$, we have
	\[
	\begin{split}
		\|\Phi(v)\|_{L^\infty([0,T],M_{2,1})} &\lesssim TR^3 + TR\langle T\rangle\|u_0\|^2_{M_{4n+2,1}} + T^{2n-1}R\langle T\rangle^{2n-1}\|u_0\|_{M_{4n+2,1}}^{4n-2}\\
		&\qquad + T^n\langle T\rangle^{n+1/2}\|u_0\|_{M_{4n+2,1}}^{2n+1} + T^{3n-2}\langle T\rangle^{3n-3/2}\|u_0\|_{M_{4n+2,1}}^{6n-3}.
	\end{split}
	\]
	Choosing $T \lesssim \min(1,\|u_0\|_{M_{4n+2,1}}^{-2})$ and $R \sim \|u_0\|_{M_{4n+2,1}}$ makes $\Phi$ into a mapping
	\[
		\Phi: \{\|v\|_{L^\infty([0,T],M_{2,1})} \leq R\} \to \{\|v\|_{L^\infty([0,T],M_{2,1})} \leq R\}.
	\]
	Since we can obtain a similar estimate on $\Phi(v_1)-\Phi(v_2)$ via polarization, this shows that we can employ the Banach fixed point argument to get a unique solution $v \in L^\infty([0,T],M_{2,1})$ of \eqref{eq:nlsforv}. Since we could have iterated in $C^0([0,T],M_{2,1})$ as well, we obtain continuity of $v$.
	
	To prove the stronger blow-up criterion, if $\|v(T)\|_{M_{4n+2,1}}$ stays bounded close to $T^*$, then we can use $\tilde u(T) + v(T) \in M_{4n+2,1}$ as new initial data for NLS. But then we transform this into an equation for $v$ again and obtain a small $\delta > 0$ such that we can solve \eqref{eq:nlsforv} on $[T,T + \delta]$ with $T + \delta > T^*$, yielding a contradiction to the maximality.
	
	For the stronger uniqueness statement we note that we can also construct a unique solution $u$ of NLS in $L^\infty([0,T],M_{4n+2,1})$ directly due to its algebra property. Since $u$ and $v$ only differ by finitely many terms which do not blow up in $M_{4n+2,1}$, the uniqueness from $u$ transfers too.
\end{proof}

To go from local to global we need to bound a controlling norm for large times. Our controlling norm will we the $H^1$ norm and the way to bound it will be via estimating the derivative of the time-dependent Hamiltonian and using a Gronwall argument. Since we need the Hamiltonian to control the energy, the method only applies in the defocusing case. This method has also been used in \cite{schippa} as well as in \cite{klaus2021} to prove global wellposedness of NLS equations in $H^{1}(\R) + H^{s}(\T)$, and it proves to be valuable here as well. More precisely, the difference NLS equation \eqref{eq:nlsforv2} is Hamiltonian with respect to
\begin{equation*}
    H(t,v) = \int \frac{1}{2} |v_x|^2 + \frac{1}{4}\big(|v+\tilde u(t)|^{4} - |\tilde u(t)|^{4} - 4\real(\bar vG(t)) \big) \, dx.
\end{equation*}
From the embedding $H^1 \subset M_{2,1} \subset M_{4n+2,1}$ and Lemma \ref{lemma:localmodulation} we see that a bound on the $H^1$ norm suffices to upgrade our local to a global result. Arguing as in Lemma \ref{lemma:gwprightup1}, we find that if we start with one more derivative, i.e. take $u_0 \in M^1_{4n+2,1}$, then the same holds for the solution $u$.

We first show that when adding an $L^2$ norm, the Hamiltonian is strong enough to control the $H^1$ norm:
\begin{lemma}
	For all $T > 0$ and $u_0 \in M_{4n+2,1}$ there exists a constant $C > 0$ such that
	\begin{equation}\label{eq:equivalentenergies}
		E(v)+\|v\|_{L^2}^2 \lesssim H(t,v) + \|v\|_{L^2}^2+1 \lesssim E(v) + \|v\|_{L^2}^2 + 1,
	\end{equation}
	where
	\[
		E(v) = \int \frac{1}{2} |v_x|^2 + \frac{1}{4}|v|^4\,dx.
	\]
	The constant depends on $n$, $\|u_0\|_{M_{4n+2,1}}$ and $T$.
\end{lemma}
\begin{proof}
	For $0 \leq t \leq T$,
	\[
	\begin{split}
		\int |v+\tilde u|^{4} - |v|^4 - |\tilde u|^{4} - &4\real(|\tilde u|^2 \tilde u\bar v)\,dx \\
		&\leq c\int |v|^2|\tilde u|(|v|+|\tilde u|)\,dx\\
		&\leq c(\|\tilde u\|_{L^\infty}^2 \|v\|_{L^2}^2 + \|\tilde u\|_{L^\infty}\|v\|_{L^3}^3)\\
		&\leq c(\|\tilde u\|_{L^\infty}^2 \|v\|_{L^2}^2 + \|\tilde u\|_{L^\infty}\|v\|_{L^2}\|v\|_{L^4}^2)\\
		&\leq (1+(C(\varepsilon))\|\tilde u\|_{L^\infty}^2 \|v\|_{L^2}^2 + \varepsilon E(v).
	\end{split}
	\]
	This term is fine due to the estimate $\|\tilde u\|_{L^\infty_{t,x}} \lesssim_T \|u_0\|_{M_{4n+2,1}}$. Knowing that
	\[
		\int |\tilde u|^4\,dx \leq C(\|u_0\|_{M_{4n+2,1}},T),
	\]
	it remains to show that $|\tilde u|^2 \tilde u - G(t)$ can be estimated in $L^2$ if $u_0 \in M_{4n+2,1}$. Indeed, we rewrite it as
	\[
	\begin{split}
		|\tilde u|^2 \tilde u &= \sum_{k_1,k_2,k_3=1}^{2n-1} A_{k_1}(u_0)\bar A_{k_2}(u_0)A_{k_3}(u_0) \\
		&= \sum_{k=1}^{2n-1} \sum_{k_1 + k_2 + k_3 = k} A_{k_1}(u_0)\bar A_{k_2}(u_0)A_{k_3}(u_0) + R(t) = G(t) + R(t),
	\end{split}
	\]
	where $R(t)$ has only terms of homogeneity $2n+1 \leq k \leq 6n-3$. Thus as in the proof of Lemma \ref{lemma:localmodulation}, for all $T > 0$,
	\[
		\|R\|_{L^\infty([0,T],L^2)} \lesssim_T \|u_0\|^{2n+1}_{M_{4n+2,1}} + \|u_0\|^{6n-3}_{M_{4n+2,1}}.
	\]
	Hence
	\[
	\begin{split}
		\int \real((|\tilde u|^2\tilde u - G(t))\bar v)\,dx &\lesssim_T \|v\|_{L^2}(\|u_0\|^{2n+1}_{M_{4n+2,1}} + \|u_0\|^{6n-3}_{M_{4n+2,1}})\\
		&\leq \|v\|_{L^2}^2 + C(\|u_0\|_{M_{4n+2,1}}),
	\end{split}
	\]
	which implies \eqref{eq:equivalentenergies}.
\end{proof}

\begin{theorem}\label{thm:gwpplarge}
	Let $2 < p < \infty$ and assume that $u_0 \in M_{p,1}^1$. Then the local solution from Lemma \ref{lemma:localmodulation} exists for all times. In particular, there exists a unique global solution $u \in C^0([0,\infty),M_{p,1}^1)$ to the defocusing cubic NLS with initial data $u(0) = u_0$.
\end{theorem}

\begin{proof}
	Via scaling (see e.g. Theorem 3.2. in \cite{corderookoudjou}) we reduce to consider small initial data. Moreover, there exists an $n \in \N_0$ such that $p \leq 4n+2$, hence Lemma \ref{lemma:localmodulation} is applicable and without loss of generality we may assume $p = 4n+2$. Fix some $T > 0$.
	
	We look at the time derivatives of the $L^2$ norm and $H$ and aim to use Gronwall. Now with the notation
	\[
		(f,g) = \int \real(f\bar g)\,dx,
	\]
	we calculate that for $0 \leq t \leq T$,
	\begin{align*}
		\partial_t \frac{1}{2}\|v\|_{L^2}^2 &= (v,v_t) = \big(v,|v+\tilde u|^2(v+\tilde u)-G(t)\big)\\
		&\lesssim \int |v|^2(|v|^2+|\tilde u|^2) + |v|(|\tilde u|^2 \tilde u - G(t))\,dx\\
		&\lesssim E(v) + \|\tilde u\|^2_{L^\infty([0,T]\times \R)}\|v\|_{L^2}^2 + \|v\|_{L^2}\||\tilde u|^2 \tilde u - G(t)\|_{L^\infty([0,T],L^2)}\\
		&\lesssim E(v) + \|v\|_{L^2}^2 +1.
	\end{align*}
	The last inequality was proven in the proof of \eqref{eq:equivalentenergies} and its constant depends both on $T$ and $\|u_0\|_{M_{4n+2,1}}$. For the Hamiltonian, we argue as in \cite[Theorem 4.1]{klaus2021} to see that only time derivatives on terms with $\tilde u$ and $G$ prevail,
	\begin{equation}\label{eq:twosummands}
		\partial_t H = (\tilde u_t,|v|^{2}v+|v|^2\tilde u+ 2\real(\bar v \tilde u)v) + \big(v,\partial_t(|\tilde u|^2\tilde u- G)\big).
	\end{equation}
	Indeed, for the bilinear part of $H$ we calculate \footnote{Strictly speaking this is only formal, the term $(v_t,-v_{xx})$ is not well-defined because both factors are only distributional. One can make this rigorous by going to the interaction picture in the calculation, see \cite[Theorem 4.1]{klaus2021} for details.}
	\begin{align*}
		\partial_t \frac12 (v_x,v_x) &= (v_t,-v_{xx}) = -(v_t,|v+\tilde u|(v+\tilde u)-G),
	\end{align*}
	and for the remaining part,
	\begin{align*}
		&\partial_t \int \frac{1}{4}\big(|v+\tilde u|^{4} - |\tilde u|^{4}\big) - \real(\bar vG)  \, dx 
		\\
		&\qquad = (v_t, |v+\tilde u|^2(v+\tilde u)-G) + (\tilde u_t, |v+\tilde u|^2(v+\tilde u)-|\tilde u|^2 \tilde u) - (v,G_t),
	\end{align*}
	from which \eqref{eq:twosummands} follows. We recall $\tilde u_t = -iG(t) + i\tilde u_{xx}$ and plug this into the first summand. The worst term is
	\[
		(\tilde u_{xx},|v|^2 v) = -(\tilde u_x,(|v|^2v)_x) \lesssim \|\tilde u_x\|_{L^\infty_{t,x}} \|v\|_{L^4}^2\|v_x\|_{L^2} \lesssim E(v),
	\]
	since we are able to bound $\tilde u_x$ in $L^\infty$ because $u_0 \in M_{4n+2,1}^1 \subset M_{\infty,1}^1$. Since $G$, $\tilde u$, and $\tilde u_x$ can be bounded in $L^\infty$ uniformly in time, the other terms in the first summand of \eqref{eq:twosummands} are estimated more easily. It remains to estimate
	\[
		\big(v,\partial_t(|\tilde u|^2\tilde u- G)\big) = (v,\partial_t R),
	\]
	where with the notation from the proof of \eqref{eq:equivalentenergies} we have
	\[
		R = \sum_{k_i, k_1 + k_2 + k_3 \geq 2n+1}^{2n-1} A_{k_1}(u_0)\bar A_{k_2}(u_0)A_{k_3}(u_0).
	\]
	Now for each $k$,
	\[
		i\partial_t A_k(u_0) + \partial_x^2 A_{k}(u_0) = \sum_{k_1+k_2+k_3 = k} A_{k_1}(u_0)\bar A_{k_2}(u_0)A_{k_3}(u_0).
	\]
	Again the worst term comes from the two derivatives. From partial integration,
	\[
	\begin{split}
		&(v,(\partial_x^2 A_{k_1})\bar A_{k_2} A_{k_3}) \\
		&\qquad = - (v_x,(\partial_x A_{k_1})\bar A_{k_2} A_{k_3}) - (v, (\partial_x A_{k_1})(\partial_x \bar A_{k_2}) A_{k_3}) - (v, (\partial_xA_{k_1})\bar A_{k_2}\partial_x A_{k_3}).
	\end{split}
	\]
	In order to use Cauchy Schwartz we have to be sure that the functions that are integrated against $v$ or $v_x$ are in $L^2$. But this holds true since $k_1 + k_2 + k_3 \geq 2n+1$ and since $u_0 \in M_{4n+2,1}^1$. All in all we find
	\[
		\partial_t (H + C\|v\|_{L^2}^2) \lesssim H + C\|v\|_{L^2}^2+1 \quad \text{for all}\quad 0 \leq t \leq T,
	\]
	and hence by Gronwall's lemma
	\[
		\sup_{t \in [0,T]} H(t,v) + C\|v\|_{L^2}^2 < \infty,
	\]
	which proves the theorem.
\end{proof}

\begin{remark}\label{rem:gwp}
	The same method applies to $M_{p,q}^s$ for $2 < p < \infty$ and $s > 2 - 1/q$. In this case by Theorem \ref{thm:embedding} an embedding $M_{p,q}^s \subset M^1_{p,1}$ holds so that the local wellposedness becomes trivial by the algebra property. See also \cite{schippa} for $p = 4$ and $p = 6$, and the remark therein for general $p$ and $q = 2$. This shows that for all spaces $M_{p,q}^s$ with $2 \leq p < \infty, 1\leq q \leq \infty$ one has global wellposedness if $s$ is large enough. Using Lemma \ref{lemma:gwprightup1} and Theorem \ref{thm:gwpp=2} the same holds true for $1 \leq p \leq 2$. It remains open whether a global result can be achieved in a space $M^s_{p,q}$ with $p = \infty$.
\end{remark}

\section{Illposedness for negative regularity}\label{sec:illposedness}

We complement the wellposedness results and show that the cubic NLS is not quantitatively well-posed in $M^s_{p,q}$ if $s<0$. This includes the cases $p,q = \infty$ and extends considerations from the introduction of \cite{schippa} where illposedness was shown using Galilean invariance. We want to remark that results on norm-inflation for nonlinear Schrödinger equations in modulation spaces have been proven in \cite{MR4165061}, though some of them rule out the cubic case due to the complete integrability. The proof of our result is inspired by \cite{MR1885293}. More precisely, we show that:

\begin{theorem}
    When $s < 0$, there is no function space $X_T$ which is continuously embedded into $C([0,T],M^s_{p,q}(\R))$ such that there exists a $C>0$ with
    \begin{equation}
        \|S(t)f\|_{X_T} \leq C \|f\|_{M^s_{p,q}},
    \end{equation}
    and
    \begin{equation}
        \|\int_0^t S(t-s)|u|^2u (s) \, ds\|_{X_T} \leq C \|u\|_{X_T}^3.
    \end{equation}
    In particular, there is no $T>0$ such that the flow map $f \mapsto u(t)$ mapping $f$ to a unique local solution on the interval $[-T,T]$ is $C^3$ at $f=0$ from $M^s_{p,q}$ to $M^s_{p,q}$.
\end{theorem}
\begin{proof}
    We first prove that the failure of the above estimates implies that the data-to-solution map cannot be $C^3$. Indeed, we consider $f = \gamma u_0$ where $u_0 \in M_{p,q}^s$ is fixed, and denote by $u(\gamma,t,x)$ the unique solution of \eqref{eq:fpnls}. Then
    \begin{equation*}
    \begin{split}
		u &= S(t)\gamma u_0 \mp 2i \int_0^t S(t-s)(|u|^2 u)\,ds,\\
		\partial_\gamma u &= S(t)u_0 \mp 2i \int_0^t S(t-s)(2|u|^2 \partial_\gamma u + u^2 \partial_\gamma\bar u)\,ds,\\
		\partial_\gamma^2 u &= \mp 2i \int_0^t S(t-s)(2|u|^2 \partial_\gamma^2 u + u^2 \partial_\gamma^2\bar u + 4|\partial_\gamma u|^2 u + 2(\partial_\gamma u)^2 \bar u)\,ds,\\
		\partial_\gamma u &= \mp 2i \int_0^t S(t-s)(2|u|^2 \partial_\gamma^3 u + u^2 \partial_\gamma^3\bar u + 6\partial_\gamma^2 u \partial_\gamma u \bar u + 6 \partial_\gamma^2 u \partial_\gamma \bar u u + 6 \partial_\gamma^2 \bar u \partial_\gamma u u \\
		&\quad+ 6|\partial_\gamma u|^2 \partial_\gamma u)\,ds.
	\end{split}
    \end{equation*}
    Putting $\gamma  = 0$ will give $u = 0$, then $\partial_\gamma u = S(t)u_0$, then $\partial_\gamma^2 u = 0$ and,
    \[
        \partial_\gamma^3 u(0,t,x) = \mp 12i \int_0^t S(t-s) (|S(s)u_0|^2 S(s)u_0) \, ds.
    \]
    If the flow is $C^3$, then this implies for any $t \in [0,T]$ the bound
    \begin{equation}\label{eq:1stiterate}
        \|\int_0^t S(t-s) (|S(s)u_0|^2 S(s)u_0) \, ds\|_{M^s_{p,q}} \lesssim \|u_0\|_{M^s_{p,q}}^3.
    \end{equation}
    We will show below that \eqref{eq:1stiterate} fails, which then gives the claim.
    
    To show that there is no quantitative wellposedness, we show failure of \eqref{eq:1stiterate} as well. Indeed, using the linear bound in the nonlinear bound would exactly imply \eqref{eq:1stiterate}.
    
    To prove failure of \eqref{eq:1stiterate}, we look for a lower bound of
    \begin{align*}
        g(t,x) = \int_0^t S(t-s) (|S(s)u_0|^2 S(s)u_0) \, ds.
    \end{align*}
    Denote by $\hat{g}(t,\xi)$ the Fourier transform $x \mapsto \xi$ of $g$. We rewrite
    \begin{align*}
        \hat g(t,\xi) &= \int_0^t e^{i(t-s)\xi^2} \int_{\xi_1 - \xi_2 + \xi_3 = \xi} e^{is(\xi_1^2 - \xi_2^2 +\xi_3^2)}\hat u_0(\xi_1) \hat u_0(\xi_2) \hat u_0(\xi_3) \, d\xi_1 d\xi_3 ds\\
        &= e^{it\xi^2} \int_{\xi_1 - \xi_2 + \xi_3 = \xi}\hat u_0(\xi_1) \hat u_0(\xi_2) \hat u_0(\xi_3) \frac{e^{it\chi} -1}{i \chi}\,d\xi_1 d\xi_3,
    \end{align*}
    where $\chi = \xi_1^2 - \xi_2^2 + \xi_3^2 - \xi^2$. We choose $\hat u_0(\xi) = \chi_{[N,N+\alpha]}$ compactly supported at frequency $N$, where $N \gg 1$, $\alpha \ll 1$. Then, $\hat g$ can only be nonzero when $\xi \in [N-\alpha,N+2\alpha]$. Moreover, when $\xi = \xi_1 - \xi_2 + \xi_3$, we have the factorization
    \begin{equation}
        \chi = -2(\xi-\xi_1)(\xi-\xi_3),
    \end{equation}
    which is of size $\alpha^2$. In particular choosing $\alpha \sim N^{-\varepsilon}$, we find
    \begin{equation*}
        \left| \frac{e^{it\chi} -1}{i \chi} \right| \gtrsim |t| + O(N^{-\varepsilon}).
    \end{equation*}
    Now the modulation space norm in $M_{p,q}^s$ of $u_0$ is
    \begin{equation*}
        \|u_0\|_{M_{p,q}^s} \sim N^s.
    \end{equation*}
    Similarly, by integrating in $\xi_1$ and $\xi_3$, the modulation space norm of $g$ is then
    \begin{equation*}
        \|g\|_{M_{p,q}^s} \gtrsim N^s\alpha^2.
    \end{equation*}
    This shows that in order for \eqref{eq:1stiterate} to hold, we need to have
    \begin{equation*}
        N^{s-\varepsilon} \lesssim N^{3s}
    \end{equation*}
    which only works when $s \geq 0$.
\end{proof}

\newpage

\bibliographystyle{plain}

\end{document}